%% file: 0paper.tex
\documentclass{amsart}
\usepackage[english]{babel} 
\usepackage[utf8]{inputenc}
\usepackage[T1]{fontenc}

\usepackage{amsmath,amsfonts,amssymb,amsthm}
\usepackage{bbm,mathrsfs,MnSymbol}
\usepackage{xfrac}
\usepackage[longnamesfirst,numbers]{natbib}
\usepackage[margin]{fixme}
\usepackage{setspace}
\usepackage{enumitem}         
\usepackage{lscape}
\usepackage{stmaryrd} 
\usepackage{todonotes}
\usepackage{color}
\definecolor{darkgreen}{rgb}{0,0.55,0}
\usepackage{hyperref}
\hypersetup{pdfstartview={Fit},    
  pdftitle={Interface approximation},    
  pdfauthor={Marvin S. Mueller},     
  pdfsubject={},   
  pagebackref=true,
  colorlinks=true,
  linkcolor=blue,
  citecolor=darkgreen}

\newtheorem{thm}{Theorem}[section]
\newtheorem{prop}[thm]{Proposition}
\newtheorem{cor}[thm]{Corollary}
\newtheorem{lem}[thm]{Lemma}

\theoremstyle{definition}
\newtheorem{defn}[thm]{Definition}
\newtheorem{ass}[thm]{Assumption}
\theoremstyle{remark}

\newtheorem{rmk}[thm]{Remark}
\newtheorem{example}[thm]{Example}

\newtheorem*{notation}{Notation}
\numberwithin{equation}{section}

\input{makrosMathBasics}

\input{makrosPartialDerivatives}

\input{makrosStochastic}

\input{makrosSemigroups}

\newcommand\HS{\operatorname{HS}}

\renewcommand\H{\mathfrak H}
\renewcommand\L{\mathfrak L}

\newcommand\Lip{C_{\operatorname{Lip}}}
\newcommand\LipBd{C_{b,\operatorname{Lip}}}
\newcommand\LipLoc{C_{\operatorname{Lip},loc}}

\newcommand{\bmu}{{\overline \mu}}
\newcommand{\bsigma}{{\overline \sigma}}
\newcommand{\VI}{\mathcal{I}}
\newcommand{\llbrak}{\llbracket}
\newcommand{\rrbrak}{\rrbracket}

\newcommand{\bnabla}{\overline{\nabla}}
\title[Approximation of the interface condition for SMBPs]{Approximation of the interface condition for stochastic Stefan-type problems}
\author{Marvin S. M\"uller}
\address{Department of Mathematics, ETH Z\"urich, Switzerland}
\email{marvin.mueller@math.ethz.ch}

\thanks{The author acknowledges support by the Swiss National Science Foundation through grant SNF $205121\_163425$.}


\date{\today}
\begin{document}
\begin{abstract}
\input{abstract}
\end{abstract}

\maketitle
\allowdisplaybreaks
\input{intro}
\input{mainresults}
\input{cdc_see}
\input{approx}
\begin{appendix}
  \input{notation}
\end{appendix}
\bibliographystyle{plain}
\bibliography{literature}
\end{document}

%% file: makrosMathBasics.tex
\renewcommand\*{\cdot}
\newcommand\1{\mathbf{1}}

\newcommand\R{\mathbb{R}}
\newcommand\C{\mathbb{C}}
\newcommand\N{\mathbb{N}}

\newcommand\Q{\mathbb{Q}}

\providecommand{\abs}[1]{\left\lvert#1\right\rvert}
\providecommand{\norm}[2]{\left\lVert#1\right\rVert_{#2}}
\providecommand{\snorm}[2]{\left[#1\right]_{#2}}

\newcommand{\ip}[3]{\left\langle #1, #2\right\rangle_{#3}}

\newcommand{\RR}{\mathbb{R}}

\newcommand{\PP}{\mathbb{P}}

\newcommand{\FF}{\mathbb{F}}

\newcommand{\sL}{\mathscr{L}}
\newcommand{\sS}{\mathscr{S}}

%% file: makrosPartialDerivatives.tex
\newcommand\ddt{\tfrac{\partial}{\partial t}}

\newcommand\ddx{\tfrac{\partial}{\partial x}}
\newcommand\ddxx{\tfrac{\partial^2}{\partial x^2}}
\newcommand\ddy{\tfrac{\partial}{\partial y}}

\newcommand\ddz{\tfrac{\partial}{\partial z}}


%% file: makrosStochastic.tex
\newcommand\F{\mathcal{F}}

\renewcommand\d{\,\operatorname{d}\hspace{-0.05cm}}

\renewcommand{\P}[1]{\mathbb{P}\left[#1\right]}                     
\newcommand{\E}[1]{\mathbb{E}\left[#1\right]}                     


%% file: makrosSemigroups.tex
\newcommand\Id{\operatorname{Id}}                               
\newcommand\dom{\mathcal{D}}                                      

%% file: abstract.tex
We consider approximations of the Stefan-type condition by imbalances
of volume closely around the inner interface and study convergence of the solutions of the corresponding semilinear stochastic moving boundary problems.
After a coordinate transformation, the problems can be reformulated as stochastic evolution equations on fractional power domains of linear operators. Here, the coefficients might fail to have linear growths and might be Lipschitz continuous only on bounded sets. We show continuity properties of the mild solution map in the coefficients and initial data, also incorporating the possibility of explosion of the solutions.

%% file: intro.tex
\section*{Introduction}

We study convergence of the local solutions of semilinear
stochastic moving boundary problems under perturbation of the
interface condition and continuity in coefficients of the mild
solution map for the corresponding systems of stochastic evolution
equations on fractional power domains of sectorial operators. 

In 1888, Josef Stefan~\cite{stefan1888theorie} proposed a model for
the temperature evolution in a system of water and ice. A key
ingredient is to model the time evolution of the spatial position of
the interface between water and ice proportionally to the local
imbalance of heat flux. In one space dimension, the equations for the
evolution of the temperature $v(t,x)$ at time
$t$ and space position $x$, and the position of the interface between water and
ice $p^*$ read as 
\begin{gather}\label{eq:Stefan_problem}
  \begin{split}
    \ddt v(t,x) &= \eta_+ \ddxx v(t,x),\qquad x>p^*(t),\\
    \ddt v(t,x) &= \eta_- \ddxx v(t,x),\qquad x<p^*(t),\\
    \ddt p^*(t) &= \varrho \cdot (\ddx v(t,p^*(t)-) - \ddx v(t,p^*(t)+)),\\
    v(t,p^*(t)) &= 0,
  \end{split}
\end{gather}
where $\eta_+$ and $\eta_-$ are diffusion coefficients inside the
respective phases and $\varrho>0$ is a proportionality constant.

Recently, semilinear and stochastic extensions
of the Stefan problem~\eqref{eq:Stefan_problem} have been studied in the context of
demand and supply modeling in modern financial markets
\cite{hambly2018reflected, keller2016stefan, mueller2016stochastic, zheng2012stochastic}, where trading works fully electronic via so called limit order books. In this
framework, $x\in \R$ describes a price level (e.\,g. in logarithmic
scale or for short time also linear scale) and $v(t,x)$ denotes the number of active buy or sell orders at time $t$ and the price level $x$. Here, we use the convention that buy orders have a negative, and
sell orders a positive sign. Demand and supply are cleared instantaneously when the price
levels of orders are matching, and so there is a price level $p^*$
separating buy and sell side of $v$. It was shown
in~\cite{keller2018forward} that under reasonable assumptions on the
coefficients one gets in fact that 
\[ v(t,x) \leq 0,\;\text{ for } x<p^*(t),\quad v(t,x) \geq
  0,\;\text{ for } x>p^*(t).\]
In a macroscopic model, $v$
now describes the density of limit orders in the order book. Then,
$p^*$ will be the so called mid-price.  

It was shown
empirically, that price changes are proportional to local
imbalances of orders placed close to the
mid-price~\cite{cont2014price, lipton2014trading}. 
For instance, a commonly used predictor for the next price move is the
volume imbalance $\VI$, which denotes the difference of limit orders
at the best buy and the best sell level. In the model, this reads as
\begin{equation*}
  \ddt p^*(t) = \varrho(\VI_t), \;t>0,
\end{equation*}
for some monotone transformation function $\varrho\colon \R\to \R$. 
Since $v$ describes the density of orders, the volume imbalance\footnote{In the empirical literature on often normalizes the volume imbalance to a value between $-1$ and $1$. As a simplification, we work with the absolute imbalance, here.} becomes
\begin{multline*}
  \VI_t:= \int_{p^*(t)}^{p^*(t)-\delta} v(t,x)\d x - \int_{p^*(t)}^{p^*(t)+\delta} v(t,x)\d x \\
  =\delta\left(-v(t,p^*(t)-) - v(t,p^*(t)+)\right) \\+  \delta^2/2\left(\ddx v(t,p^*(t)-) - \ddx v(t,p^*(t)+)\right) + o(\delta^2),
\end{multline*}
where $\delta>0$ is the minimal distance between two price levels (also called tick size). From macroscopic modeling perspective, we switch to a continuous in price-scale and are then interested to understand what happens when $\delta \searrow 0$. 

Assuming that Dirichlet boundary conditions are satisfied at $p^*$, we thus get with proper rescaling 
\begin{equation*}
  \lim_{\delta\searrow 0} \frac1{\delta^2} \VI_t =\frac12( \ddx v(t,p^*(t)-) - \ddx v(t,p^*(t)+)).
\end{equation*}
On that way, we recapture the Stefan condition for the price dynamics which has been widely used in the literature. For a
more detailed description of electronic markets using models with
Stefan-type condition for the price dynamics we refer to
\cite{mueller2016stochastic}, \cite{zheng2012stochastic} and
\cite{hambly2018reflected}.

In the following we will make the motivation to use the Stefan-type
condition mathematically rigorous by studying the convergence of the
respective densities of orders. More precisely, we analyze the
convergence of the solutions of stochastic moving boundary problems
driven by spatially colored noise and with inner boundary dynamics governed by
\begin{equation*}
  \d p^*_\delta (t) = \varrho(\delta^{-2} \VI_t)\d t, \quad \text{as } \delta \searrow 0,
\end{equation*} 
to the solutions of the respective stochastic Stefan-type problems. In
recent work~\cite{hambly2018reflected, hambly2018stefan}, Hambly and
Kalsi showed existence, uniqueness and regularity for stochastic moving boundary
problems which are driven by space-time white noise and cover
interface dynamics of Stefan and also volume imbalance type. The latter was
motivated by approximations of the Stefan condition, but convergence of the solutions has not been discussed so far.

In the next section we present our precise setup and state the convergence results. To prove the convergence we use the reformulation of stochastic moving boundary problems as stochastic evolution equations, which was
introduced in \cite{keller2016stefan} to show existence and uniqueness of solutions. In Section~\ref{sec:see}, we extend this abstract setting by studying continuous dependence of the mild solution map on
the coefficients and initial data. To overcome the issue that the solutions might explode in finite time we truncate the coefficients as in~\cite{keller2016stefan, keller2018forward} and need to deal with convergence results on the respective explosion times, following~\cite{keller2018forward}. In Section~\ref{sec:proofs}, we then apply the abstract setting to the stochastic moving boundary problems and finish the proof of the statements from Section~\ref{sec:main}. In the same way, the results on stochastic evolution equations can also be used for more general approximations of the coefficients. Relevant notation is listed in Appendix~\ref{sec:notation}.


%% file: mainresults.tex
\section{Stochastic Stefan-type problems and approximations}
\label{sec:main}
Let $\mu_+$, $\mu_-\colon \R^3\rightarrow \R$, $\sigma_+$, $\sigma_-\colon
\R^2\rightarrow \R$, $\varrho\colon \R^2\to \R$, $\eta_+$, $\eta_->0$ and $\kappa_+$, $\kappa_-\in
[0,\infty)$, let $(\Omega, \F, (\F_t)_{t\in[0,\infty)}, \PP)$ be a
stochastic basis on which exists an $\Id_{L^2(\R)}$-cylindrical Wiener
process $W$, let $\zeta \colon \R^2\to \R$ be an integral kernel such that
\begin{equation*}
  \xi_t(x) :=  T_\zeta W_t (x),\qquad T_\zeta w(x) := \int_\R \zeta(x,y) w(y) \d y,\quad x\in \R,\;t\geq 0,\;w\in L^2(\R),
\end{equation*}
defines a Brownian motion $(\xi_t(x))_{t\geq 0}$ for each $x\in \R$. 

Under additional assumptions, which will be stated below, we consider
the stochastic moving boundary problem in one space dimension, for
$t>0$, $x\in \R$, $n\in \bar \N$,
\begin{gather}
  \begin{split}
    \d v_{n}(t,x) &= \left[\eta_+ \Delta v_{n} (t,x)+
      \mu_+\left(x-p_{n}^*(t), v_{n}(t,x), \nabla v_{n}(t,x) \right) \right] \d t\\
    & \qquad \qquad\qquad\qquad+ \sigma_+\left(x-p_{n}^*(t), v_{n}(t,x)\right)\d \xi_t (x), \quad  x > p_{n}^*(t),\\
    \d v_{n}(t,x) &= \left[\eta_- \Delta v_{n}(t,x) +\mu_-\left(x-p_{n}^*(t),v_{n}(t,x), \nabla v_{n}(t,x) \right) \right] \d t\\
    &\qquad\qquad\qquad\qquad + \sigma_-\left(x-p_{n}^*(t), v_{n}(t,x)\right) \d \xi_t(x), \quad  x < p_{n}^*(t),\\
  \end{split}\label{eq:smbp}
\end{gather}
with Dirichlet boundary conditions
\begin{gather}
  v_{n}(t,p_{n}^*(t)+) =
  v_{n}(t,p_{n}^*(t)-) = 0,
  \label{eq:bc}
\end{gather}
and interface dynamics, if $n<\infty$,
\begin{equation}
  \label{eq:Stefan_cond_approx}
  \d p_{n}^*(t) = \varrho\Big(2n^2 \int_0^{1/n} v_{n}(t,p_n^*(t) + y) \d
  y, 2n^2\int_0^{1/n} v_{n}(t,p_n^*(t) - y)\d y\Big)\d t,
\end{equation}
and, if $n =\infty$,
\begin{equation}
  \label{eq:Stefan_cond}
  \d p_{\infty}^*(t) = \varrho\Big(\nabla v_{\infty}(t,p_{\infty}^*(t)+), \nabla v_{\infty}(t,p_{\infty}^*(t)-)\Big)\d t.
\end{equation}
Note that due to the Dirichlet boundary conditions the scaling $n^2$ in \eqref{eq:Stefan_cond_approx} ensures a non-trivial limit of the terms when $n\to\infty$.

We recall the notion of solution as introduced
in~\cite{keller2016stefan} for the problem with $n = \infty$. First, to formalize the moving frame for the free boundary problem, we define for each $x \in \RR$ the function space
\begin{equation*}
\Gamma(x) := \left\{v\in L^2(\R)\colon \left.v\right|_{(-\infty,x)} \in
  H^2(-\infty,x), \left.v\right|_{(x,\infty)} \in H^2(x,\infty),
  \,v(x\pm) = 0 \right\}.
\end{equation*}
Due to Sobolev embeddings, each $f\in \Gamma(x)$ admits a
uniformly continuous and bounded representative which is piece-wise
continuously differentiable. In particular, $\Gamma(x)\subset
H^1(\R)$, for each $x\in \R$. For the definition of a solution, we stress the definition of $\nabla$ and $\Delta$ as \emph{piece-wise} weak derivatives, see Appendix~\ref{sec:notation}.

\begin{defn}\label{df:smbp}
  Let $d\in \N$, $D\colon \R\times L^2(\R) \to \R$, $\bmu: \R^4\rightarrow \R$, and $\bsigma: \R^2\rightarrow \R$ be Borel measurable. A \emph{local (strong) solution} of the stochastic moving boundary problem
  \begin{equation*}
    \begin{aligned}
      \d v(t,x) &= \bmu(x-p^*(t), v(t,x), \nabla v(t,x), \Delta v(t,x))
      \d t\\
      &\qquad\qquad\qquad \qquad+ \bsigma(x- p^*(t), v(t,x)) \d \xi_t(x),\quad t\geq
      0,\,x\neq p^*(t),\\
      \d p^*(t) &= D(p^*(t), v(t,.) \d t,\\
      v(t,p^*(t)) &=0,
    \end{aligned}
  \end{equation*}
  with initial data $v_0\in L^2(\R)$ and $p_0\in \R$, is a triple $(\tau, p^*, v)$, where $\tau$ is a predictable stopping time, and 
  \begin{equation*}
    (p^*, v): \llbrak 0, \tau \llbrak \to \bigcup_{x \in \R} \left(\{x\}  \times \Gamma(x)\right) \; \subset \; \R \times L^2(\R),
  \end{equation*}
  such that $(p^*, v, \nabla v, \Delta v)$ is an adapted and continuous process on $\R\times L^2(\R)^{\times 3}$, and it holds on $\llbrak 0,\tau\llbrak$
  \begin{align*}
    v(t,\cdot) - v_0 &= \int_0^t \bmu(\cdot-p^*(s),  v(s,\cdot), \nabla v(s,\cdot), \Delta v(s,\cdot))\d s\\
                            &\qquad\qquad+ \int_0^t \bsigma(\cdot-p^*(s), v(s,\cdot))\d\xi_s(\cdot), \\
   p_*(t)  &= p_0 +  \int_0^t D(p^*(s), v(s,.))\d s,\quad t\geq 0.
  \end{align*}
  The first equality holds true in $L^2(\R)$, where the first integral
  is a Bochner integral in $L^2(\R)$, and the second one a stochastic
  integral in $L^2(\R)$.\\
  The solution is called \emph{global}, if $\tau = \infty$ a.\,s. and the solution is called maximal if there is no solution on a larger stochastic interval.
\end{defn}

\begin{notation}\label{not:bmusigma}
  For the remainder of this paper we will use the functions $\bmu:
  \R^4\rightarrow \R$, $\bsigma: \R^2\rightarrow \R$, such that for $x,v,v',v''\in \R$,
  \begin{align*}
    \bmu(x,v,v',v'')&:=
                        \begin{cases}
                          \eta_+ v'' + \mu_+(x,v,v'), & x>0,\\
                          \eta_- v'' + \mu_-(x,v,v'), & x<0,
                        \end{cases}
    \\
    \bsigma(x,v) &:= 
                     \begin{cases}
                       \sigma_+(x,v), & x>0,\\
                       \sigma_-(x,v), & x<0.
                     \end{cases}
  \end{align*}
\end{notation}
\begin{rmk}
  At this point it becomes already visible that, even when $\mu_\pm
  \equiv 0$ and we assume that $\varrho$ and $\sigma_\pm$ are linear
  functions, the stochastic moving boundary problem~\eqref{eq:smbp} is
  non-linear. This is due to the interaction mechanism of $v$ and
  $p^*$ and will become more clear below.
\end{rmk}

We now state the main assumptions, which are the same as required for existence of the stochastic Stefan problems in~\cite{keller2016stefan}.
\begin{ass}\label{ass:mu} $\mu_+$ and $\mu_-$ are continuously differentiable and
  \begin{enumerate}[label=(\roman*)]
  \item\label{ai:mugrowths} there exist $a \in L^2(\R)$, $b$, $\tilde b \in L^\infty_{loc}(\R^2)$ such that for all $x, y, z\in \R$
    \[ \abs{\mu_\pm (x,y,z)} + \abs{\ddx \mu_\pm (x,y,z)} \leq  a(|x|) + b(y,z)\left(\abs{y} + \abs{z}\right),\]
    and
    \[ \abs{\ddy \mu_\pm(x,y,z)} +  \abs{\ddz\mu_\pm (x,y,z)} \leq \tilde b(y,z), \]
  \item\label{ai:mulip} $\mu_\pm$ and their partial derivatives (in $x$, $y$ and $z$) are locally Lipschitz with Lipschitz constants independent of $x\in \R$.
  \end{enumerate}
\end{ass}

\begin{ass}\label{ass:sigma}
  $\sigma_+$ and $\sigma_-$ are twice continuously differentiable and
  \begin{enumerate}[label=(\roman*)]
  \item\label{ai:sigmagrowths} for every multi-index $I = (i,j) \in \N^2$ with $\abs{I} \leq 2$ there exist $a_I\in L^2(\R_+)$ and $b_I \in L^\infty_{loc}(\R)$ such that
    \[  \abs{\tfrac{\partial^{\abs{I}}}{\partial x^{i}\partial y^{j}} \sigma_\pm(x,y)} \leq  \begin{cases}   a_I(|x|) + b_I(y)\abs{y}, & j = 0, \\ b_I (y), & j \neq 0, \end{cases}\]
  \item\label{ai:sigmalip} $\sigma_\pm$ and their partial derivatives (in $x$, $y$ and $z$) are locally Lipschitz with Lipschitz constants independent of $x\in \R$,  
  \item\label{ai:sigmabc} $\sigma_+$ and $\sigma_-$ satisfy the boundary condition
    \begin{equation*}
      \sigma_+(0,0)=\sigma_-(0,0) = 0. 
    \end{equation*}
  \end{enumerate}
\end{ass}

\begin{ass}\label{ass:rho}
  $\varrho:\R^2\to\R$ is locally Lipschitz continuous. 
\end{ass}
\begin{ass}
  \label{ass:zeta}
  For all $y\in \R$ it holds that $\zeta(.,y) \in C^3(\R)$ and for all
  $x\in \R$ and $i\in \{0,1,2,3\}$ that $\tfrac{\partial^{i}}{\partial x^i}\zeta(x,.)\in L^2(\R)$. Moreover, 
  \begin{equation*}
    \sup_{x\in \R} \norm{\tfrac{\partial^{i}}{\partial x^i}
      \zeta(x,.)}{L^2(\R)} <\infty,\quad i=0,1,2,3.
  \end{equation*}
\end{ass}

\begin{example}[Convolution Kernel]
  Let $\zeta(x,y) := \zeta(x-y)$, $x$, $y\in \R$. If $\zeta\in C^{\infty}(\R)\cap H^3(\R)$, 
  then Assumption~\ref{ass:zeta} is satisfied. In this case, the operator $T_\zeta$ corresponds to spatial convolution with $\zeta$.
\end{example}

To achieve global existence, we will also need the following
assumption. 
\begin{ass}
  \label{ass:global}
  Assume that the functions $(b,\tilde b)$ and $\varrho$ in Assumption~\ref{ass:mu} and
  \ref{ass:rho}, respectively, are globally bounded. Moreover, assume that there exist functions $\sigma_\pm^1\in H^2(\R_+)\cap C^2([0,\infty))$ and $\sigma_\pm^2 \in BUC^2([0,\infty))$ such that
  \[ \sigma_+(x,y) = \sigma_+^1(x) + \sigma_+^2(x)y,\qquad \sigma_-(x,y) = \sigma_-^1(x) + \sigma_-^2(x)y,\]
  for all $x$, $y\in \R$. 
\end{ass}

We are now able to state the existence result. 
\begin{thm}[Existence] 
  \label{thm:existence}
  Let Assumptions~\ref{ass:mu},~\ref{ass:sigma},~\ref{ass:rho}
and~\ref{ass:zeta} hold true, and let $p_0\in \R$ and $v_0\in \Gamma(p_0)$. Then, for each $n\in \bar \N$ there
exists an up to modifications unique strong solution $(\tau_n, p_n^*,v_n)$ of the stochastic moving boundary problem~\eqref{eq:smbp}
  with interface condition~\eqref{eq:Stefan_cond_approx},
  if $n<\infty$,
  and \eqref{eq:Stefan_cond}, if $n=\infty$.
  
  If, in addition, Assumption~\ref{ass:global} holds true, then
  $\tau_n=\infty$ a.\,s. for
  each $n\in \bar \N$.
\end{thm}

The main result on stochastic moving boundary problems is now the convergence statement, which can be found in a more precise formulation in Section~\ref{sec:proofs}.
\begin{thm}[Convergence]
  \label{thm:conv_smbp}
  Let Assumptions~\ref{ass:mu},~\ref{ass:sigma},~\ref{ass:rho}
  and~\ref{ass:zeta} hold true and let $p_0\in \R$ and $v_0\in
  \Gamma(p_0)$. For $n \in\bar \N$, let $(\tau_n, p_n^*, v_n)$ be the
  unique strong solution of~\eqref{eq:smbp} from Theorem~\ref{thm:existence}. Then, for all $t\in [0,\infty)$, in probability
  \begin{align*}
    H^1(\R)-\lim_{n\to \infty}v_{n}(t,\cdot) \1_{\llbrak 0, 
    \tau_n\wedge \tau_\infty\llbrak }(t) &= v_\infty(t,\cdot) 
                                           \1_{\llbrak 0, \tau_\infty\llbrak}(t),\\
    L^2(\R)-\lim_{n\to \infty}\Delta v_{n}(t,\cdot) \1_{\llbrak 0, 
    \tau_n\wedge \tau_\infty\llbrak }(t) &= \Delta v_\infty(t,\cdot) 
                                           \1_{\llbrak 0, \tau_\infty\llbrak}(t),\\    
    \lim_{n\to\infty} p^*_{n}(t) \1_{\llbrak 0,  \tau_n\wedge \tau_\infty\llbrak }(t)
                                         &= p^*_\infty(t) \1_{\llbrak 0, \tau_\infty\llbrak }(t).
  \end{align*}
  If, in addition, Assumption~\ref{ass:global} holds true, then for
  each $q\in [1,\infty)$ and $T\in (0,\infty)$,
  \begin{align*}
    \lim_{n\to\infty} \E{\sup_{0\leq t\leq T} \norm{v_\infty(t,\cdot)-
    v_n(t,\cdot)}{H^1(\R)}^q }&=0,\\
    \lim_{n\to\infty} \E{\sup_{0\leq t\leq T}\norm{\Delta v_\infty(t,\cdot) - \Delta v_n(t,\cdot)}{L^2(\R)}^q  }&=0,\\
    \lim_{n\to\infty} \E{\sup_{0\leq t\leq T}  \abs{p_\infty^*(t) - p_n^*(t)}^q}&=0.
  \end{align*}
\end{thm}
\begin{rmk}
  After truncation of the coefficients as will be done in an abstract
  framework below, one obtains the
  convergence rate $1/2$ for the solutions of the truncated
  equations. Hence, the convergence above is of that rate as
  long as the solutions do not get ``{}large''{}; see
  Remark~\ref{rmk:convrate} for a more precise formulation. However,
  further analysis is required to understand whether the convergence
  rate also holds for convergence of $(v_n,p_n)$ to
  $(v_\infty,p_\infty)$ without localization.
\end{rmk}

\subsection{Outline of the proof}
\label{ssec:proofmainres}
For the proof and further analysis we will, as in \cite{keller2016stefan}, reformulate~\eqref{eq:spde} in terms of stochastic evolution equations.

To this end, one considers the problems in centered coordinates, which
yields semilinear SPDEs on the fixed domain $\dot \R= \R\setminus\{0\}$. We get for
$n\in \bar \N$, $u_{n}(t,x):= v_n(t,x+p_*(t))$, $x\neq 0$, $t\geq 0$, the equation
\begin{gather}
  \begin{split}
    \d u_{n}(t,x) &= \left[\eta_+\Delta u_{n}(t,x) +  \mu_+\left(x,
        u_{n}(t,x), \nabla u_{n}(t,x) \right)\vphantom{ \varrho(\nabla
        u_{n}(t,0+),\nabla u_{n}(t,0-))\nabla u_{n}(t,x)}\right.  \\
    &\qquad\qquad\qquad \left.+\vphantom{(\eta_+) \Delta u_{n} +
        \mu_+\left(x,p_{n}^*(t), u_{n}, \nabla u_{n} \right)+}
      \varrho(\nabla u_{n}(t,0+),\nabla u_{n}(t,0-))\nabla u_{n}(t,x)\right] \d t\\
    & \qquad + \sigma_+\left(x,  u_{n}(t,x))\right)\d \xi_t (x+p_{n}^*(t)) , \quad  x > 0,\\
    \d u_{n}(t,x) &= \left[\eta_- \Delta u_{n} +\mu_-\left(x,  u_{n}(t,x), \nabla u_{n}(t,x) \right) \vphantom{ + \varrho(u_{n}(t,0+),u_{n}(t,0-),\nabla u_{n}(t,0+), \nabla u_{n}(t,0-))\nabla u_{n}(t,x)} \right.\\
    &\qquad\qquad\qquad + \left.\vphantom{(\eta_-) \Delta u_{n} +\mu_-\left(x, p_{n}^*(t),  u_{n}, \nabla u_{n} \right)+} \varrho(\nabla u_{n}(t,0+), \nabla u_{n}(t,0-))\nabla u_{n}(t,x)\right] \d t\\
    &\qquad + \sigma_-\left(x,  u_{n}(t,x)\right) \d \xi_t(x+p_{n}^*(t)), \quad  x < 0,
  \end{split}\label{eq:spde}
\end{gather}
with boundary conditions
\begin{equation*}
  u_{n}(t,0+) = u_{n} (t,0-) = 0,
\end{equation*}
and, as above, for $n <\infty$,
\begin{equation*}
  \d p_{n}^*(t) = \varrho\Big(2n^{2} \int_0^{1/n}u_{n}(t,y)\d
  y,2n^2\int_0^{1/n}  u_{n}(t,-y)\d y\Big)\d t ,
\end{equation*}
and, for $n=\infty$,
\begin{equation*}
  \d p_{\infty}^*(t) = \varrho\Big(\nabla u_{\infty}(t,0+), \nabla u_{\infty}(t,0-)\Big)\d t.
\end{equation*}

Reflecting $(-\infty,0)$ to $(0,\infty)$, we will later reformulate the
problem in terms of stochastic evolution equations on the spaces
\[ \L^2:= L^2(\R_+)\oplus L^2(\R_+) \oplus \R,\quad  \H^\alpha :=
  H^{\alpha}(\R_+) \oplus H^{\alpha}(\R_+) \oplus ,\;\alpha \in\R.\]
This provides a rich framework for analysis of the solutions.

The outline for the details is now as follows
\begin{itemize}
\item In Section~\ref{sec:see} we will discuss the solution map for
  stochastic evolution equations on a class of interpolation spaces
  with focus on its dependence on the non-linearities and the initial state
\item In Section~\ref{sec:proofs} we then apply the results from
  Section~\ref{sec:see} to the fixed boundary problem~\eqref{eq:spde}
  and discuss the transformation between fixed and moving boundary problems.
\item The notation used in this paper is given in Appendix~\ref{sec:notation}.
\end{itemize}

Finally, Theorem~\ref{thm:existence} follows from
Proposition~\ref{prop:trafo}. As well, Theorem~\ref{thm:conv_smbp}
follows by Proposition~\ref{prop:trafo} but applying also
Theorem~\ref{thm:approx_Stefan_see} and continuity of the map
$F$ which is defined in~\eqref{eq:trafofct}; see
Lemma~\ref{lem:trafo}.\ref{iv:trafo} and \ref{iii:trafo}.  \hfill \qed

\begin{rmk}
  The results we will prove in Section~\ref{sec:see} can also be used
  to show continuous dependency of~\eqref{eq:smbp} on the coefficients $\mu$,
  $\sigma$ and $\zeta$, which might be of use e.\,g. for numerical approximations.
\end{rmk}


%% file: cdc_see.tex
\section{Mild solution map for stochastic evolution equations}
\label{sec:see}
We now discuss a class of stochastic evolution equations with focus on continuous dependency of the solution
map in the coefficients of the equations. We first recall some facts
on fractional powers of linear operators and its domains. Then, we
discuss the general setting and finally the main results on continuity
of the mild solution map for stochastic evolution equations on Hilbert
spaces. 

\subsection{Preliminaries from analysis}
\label{ssec:prel}
In this subsection, let $(E, \norm{.}{E})$ be a Banach space and $A\colon
\dom(A)\subseteq E\to E$ be a densely defined and sectorial operator
with domain $\dom(A)\subset E$. Whenever necessary to apply results
from the literature which require complex Banach spaces we might
switch to the complexification without further mentioning.

We assume that the resolvent set of $A$ contains $[0,\infty)$ and there exists a $M > 0$ such that the resolvent $R(\lambda,A)$ satisfies 
\begin{equation}\label{eq:resolvent_eq}
  \norm{R(\lambda,A)}{L(E)} \le \frac{M}{1 + \lambda}, \qquad \text{for all $\lambda > 0$.}
\end{equation}
\begin{rmk}
  The conditions on $A$ are equivalent to each of the following statements
  \begin{itemize}
  \item  Equation \eqref{eq:resolvent_eq} holds, the resolvent set of $A$ contains $0$ and a sector
    \[\{\lambda \in \C: |\arg \lambda | < \theta\}\]
    for some $\theta \in (\pi/2,\pi)$. 
  \item The operator $A$ is sectorial and $-A$ is positive in the sense of \cite{lunardi2009interpolation}.
  \item $A$ is the generator of an analytic $C_0$-semigroup $(S_t)_{t \geq 0}$ of negative type. 
  \end{itemize}
  In particular, there exist $\delta$, $M > 0$ such that $\norm{S_t}{L(E)}
  \le M e^{-\delta t}$. 
\end{rmk}
The assumption ensures that fractional powers of $-A$ are well defined.
\begin{notation}
  For $\alpha \geq 0$ we write
  \begin{equation*} 
    E_\alpha :=  \dom((-A)^\alpha),\qquad \norm{x}{E_\alpha} := \norm{(- A)^\alpha x}{E}, \; x\in E_\alpha.
  \end{equation*}
  It is well-known that also $E_\alpha$ with the induced scalar
  product is a Banach space again, and when $E$ is a Hilbert space
  then so is $E_\alpha$. In particular, $\norm{.}{1}$ is equivalent to the graph norm of $A$ and the following continuous embedding relations hold for $\alpha \in [0,1]$:
  \begin{equation*}
    \dom(A) = E_1 \hookrightarrow E_\alpha \hookrightarrow E_0 = E.
  \end{equation*}
\end{notation}

We recall the following reiteration property.
\begin{prop}\label{prop:reiteration}
  Let $\alpha$, $\beta\in \R$, and $x\in E_{\alpha+\beta}\cap E_{\alpha}\cap E_{\beta}$. Then,
  \[ (-A)^{\alpha}((-A)^{\beta} x) = (-A)^{\beta}((-A)^{\alpha} x) = (-A)^{\alpha+\beta} x.\]
\end{prop}
\begin{rmk}\label{rmk:interpol:reiteration}
  For $\alpha>0$, the part of $A$ in $E_\alpha$, is again a densely defined and closed operator on $E_\alpha$. Moreover, it is the infinitesimal generator of the restriction of $S_t$ to $E_\alpha$, which is again an analytic and strongly continuous semigroup. The same holds true for the extension of $S_t$ to $E_\alpha$, when $\alpha <0$; see e.g. \cite[Ch.~II.5]{engel1999one}.
\end{rmk}
The following regularity property of $S_t$ between different
interpolation spaces $E_\alpha$, $\alpha \in [0,1]$ will be crucial in
the next section. We derive it from results in~\cite{lunardi2009interpolation} on interpolation spaces.
\begin{lem}\label{lem:StHa}
  Let $\beta \geq 0$ and $\alpha > \beta$. Then, for all $t>0$ and $x\in E_{\beta}$,
  \[ \norm{S_t x}{E_\alpha} \leq K_{\alpha,\beta} t^{\beta - \alpha} \norm{x}{E_\beta}. \]
\end{lem}
If $\beta\in (\alpha-1,\alpha)$, then the proportionality constant on the right hand side is
integrable at $t = 0$, which will be a key property used in the
following sections. To deal with this singularity we also have to understand
the convolutions below. 
\begin{lem}
  \label{lem:sup_conv}
  Let $T\in (0,\infty)$, $\alpha\in (0,1)$ and let $\phi\colon [0,T]\to [0,\infty)$ be
  Borel measurable and bounded. Then, 
  \begin{equation*}
    \sup_{0\leq t\leq T} \int_0^t \phi(s) (t-s)^{-\alpha} \d s \leq
    \int_0^T \left(\sup_{0\leq r\leq s} \phi(r)\right) (T-s)^{-\alpha} \d s.
  \end{equation*}
\end{lem}
\begin{proof}
  \begin{multline*}
    \sup_{0\leq t\leq T} \int_0^t \phi(s) (t-s)^{-\alpha} \d s =
    \sup_{0\leq t\leq T} \int_0^t \phi(t-s) s^{-\alpha} \d s \\
    \leq \sup_{0\leq t\leq T} \int_0^t \sup_{0\leq r\leq T-s} \phi(r)
    s^{-\alpha} \d s 
    =\int_0^T \sup_{0\leq r\leq s} \phi(r)
    (T-s)^{-\alpha} \d s\qedhere
  \end{multline*}
\end{proof}
We will also need the following version of
Gronwall's lemma, see~\cite[Lem 7.0.3]{lunardi1995analytic} or, for a proof~\cite[p. 188]{henry2006geometric}.
\begin{lem}[Extended Gronwall's lemma]\label{lem:gronwall}
  For all $\alpha\in(0,1)$, $b\in[0,\infty)$, $T \in [0,\infty)$ there exists
  a constant $K_{\alpha, b,T} \in [0,\infty)$ such that for all $a\in
  [0,\infty)$ and all integrable
  $\phi:[0,T]\to [0,\infty)$ which are satisfying for all
  $t\in [0,T]$,
  \begin{equation*}
    \phi(t) \leq a + b\int_0^t \phi(s) (t-s)^{-\alpha} \d s, 
  \end{equation*}
  it holds for all $t\in [0,T]$,
  \begin{equation*}
    \phi(t) \leq a K_{\alpha,b,T}.
  \end{equation*}
\end{lem}

We also keep the following basic lemma from analysis.
\begin{lem}
  \label{lem:unif_convergence}
  Let $(E,\norm{.}{E})$, $(V,\norm{.}{V})$ be Banach spaces and for $n\in \bar \N$ let $\Phi_n \in \Lip(E;V)$, be such that for all $x\in E$
  \[\lim_{n\to\infty}\norm{\Phi_\infty(x) - \Phi_n(x)}{V} = 0,\]
  and
  \[\sup_{n\in \N} \snorm{\Phi_n}{\Lip(E;V)}<\infty.\]
  Then, for all $K\subset E$ compact, it holds that
  \begin{equation*}
    \lim_{n\to\infty} \sup_{x\in K}\norm{\Phi_\infty(x) - \Phi_n(x)}{V} = 0.
  \end{equation*}
\end{lem}
\begin{proof}
  Let $K\subset E$ be compact and define
  \[ L_\Phi := \sup_{n\in \bar\N} \snorm{\Phi_n}{\Lip(E;V)}. \] 
  For $\epsilon >0$ set $\delta:=
  \sfrac{\epsilon}{(4 L_\Phi)}$ and let $N\in \N$ and $x_1,\ldots, x_N\in K$ be such
  that $K\subseteq \bigcup_{k=1}^N
  B_E(x_k,\delta)$, where
  \[ B_E((x, \delta):= \{ y\in V\colon \norm{x-y}{V} < \delta\}, \quad
    x\in V.\]
  By strong convergence of $(\Phi_n)_{n\in \N}$ we can choose $n_0\in \N$ such that for all $n\geq n_0$ it holds that
  \[\sup_{k=1,..,N} \norm{\Phi_n(x_k) - \Phi_\infty(x_k)}{V}<\sfrac{\epsilon}{2}.\]
  Hence, for all $n\geq n_0$,
  \begin{gather*}
    \begin{split}
      \sup_{x\in K} \norm{\Phi_n(x) - \Phi_\infty(x)}{V} &\leq \sup_{k=1,...,N}\sup_{x\in B_E((x_k,\delta)} \big[\norm{\Phi_n(x) - \Phi_n(x_k)}{V} \\
      &\qquad + \norm{\Phi_\infty(x) - \Phi_\infty(x_k)}{V} + \norm{\Phi_\infty(x_k) - \Phi_n(x_k)}{E}\big]\\
      &< 2 L_\Phi\delta + \epsilon/2 = \epsilon.      \qedhere
    \end{split}
  \end{gather*}
\end{proof}

\subsection{Setting}
\label{sec:sse_setting}
Let $T\in (0,\infty)$, 
let $(E, \norm{.}{E}, \ip{.}{.}{E})$ and $(U, \norm{.}{U}, \ip{.}{.}{U})$ be real separable Hilbert
spaces, 
let $(S_t)_{t\in [0,\infty)}$ be an analytic and strongly continuous
semigroup of negative type with generator $A\colon \dom(A)\subseteq E
\to E$,
and for $\alpha \in \R$ let 
\begin{equation*}
  E_\alpha := \dom((-A)^{\alpha}), \qquad \norm{x}{E_\alpha} :=
  \norm{(-A)^{\alpha}x}{E},\, x\in E_\alpha,
\end{equation*}
let $(\Omega, \F, \FF,\PP)$ be a stochastic basis, with filtration
$\FF=(\F_t)_{t\in [0,\infty)}$,
let $W$ be an $\Id_E$-cylindrical
Wiener process on $(\Omega,\F,\FF,\PP)$.

\subsection{Continuity of the solution map}
We now study continuity properties of the mild solution map for stochastic
evolution equations. Part~\ref{i:continuity_strong} can also be derived from \cite[Proposition 3.2]{kunze2012continuous} where the proof is sketched. We will
go into more details here but restrict to a framework on Hilbert spaces. This will be sufficient to cover the problems introduced above and make the proof more direct.
\begin{thm}\label{thm:approx}
  Assume the Setting~\ref{sec:sse_setting} and let $q\in
  (2,\infty)$, $\alpha \in [0,1)$.
  Then, the following holds true.
  \begin{enumerate}[label=(\roman*)]
  \item \label{i:existence} There exists a unique mapping
    \begin{equation*}
      \sS\colon L^q(\Omega, \F_0,\PP; E_{\alpha}) \times\Lip(E_\alpha; E) \times \Lip(E_{\alpha};
      \HS(U; E_{\alpha})) \to \sL^q_{\FF}(E_{\alpha})
    \end{equation*}
    which satisfies that for all $X_0\in L^q(\Omega,\F_0,\PP; E_\alpha)$, $B\in \Lip(E_{\alpha}; E)$, $C\in
    \Lip(E_{\alpha};\HS(U;E_{\alpha}))$ and all $t\in [0,T]$ it holds
    that for $X= \sS(X_0,B,C)$, almost surely,
    \begin{equation}
      \label{eq:see_mild}
      X(t)=  S_t X_0 + \int_0^t
        S_{t-s} B(X(s)) \d s + \int_0^t
      S_{t-s} C(X(s)) \d W_s.
    \end{equation}
  \item \label{i:continuity_lip} For each
    $\widetilde B\in \LipBd(E_\alpha;E)$ and $\widetilde C\in
    \LipBd(E_\alpha; \HS(U;E_\alpha))$ there exists a constant $L\in
    (0,\infty)$ such that for all $X_0$, $\widetilde X_0\in
    L^q(\Omega;E_{\alpha})$, $B \in \LipBd(E_\alpha; E)$ and $C \in \LipBd(E_\alpha; \HS(U;E_\alpha))$
    \begin{multline}
      \label{eq:sol_map_lip_unif}
      \norm{\sS(X_0, B, C) - \sS(\widetilde X_0, \widetilde B,
        \widetilde C)}{\sL^q_{\FF}(E_\alpha)}\\*
      \leq L\left(\norm{X_0 - \widetilde  X_0}{L^q(\Omega;
          E_\alpha) }
        + \norm{B- \widetilde B}{B(E_\alpha;E)} + \norm{C-\widetilde C}{B(E_\alpha; \HS(U;E_\alpha))}\right).
    \end{multline}
    In particular, the restriction of $\sS$ to
    \[L^q(\Omega, \F_0,\PP; E_{\alpha}) \times\LipBd(E_\alpha; E) \times \LipBd(E_{\alpha};
    \HS(U; E_{\alpha}))\]
    is Lipschitz continuous on bounded
    sets, globally w.\,r.\,t. $X_0$.
  \item \label{i:continuity_strong} For $X_{0,n}\in
    L^q(\Omega,\F_0,\PP; E_\alpha)$, $B_n \in \Lip(E_\alpha;E)$,
    $C_n \in \Lip(E_\alpha; \HS(U;E_\alpha))$, $n\in \bar \N$, such
    that for all $x\in E_{\alpha}$,
    \begin{align*}
      \lim_{n\to \infty}\norm{X_{0,\infty} -
      X_{0,n}}{L^q(\Omega;E_{\alpha})} &=0,\\
      \lim_{n\to \infty} \norm{B_\infty(x) - B_n(x)}{E} &=0,\\
      \lim_{n\to \infty}\norm{C_\infty(x) - C_n(x)}{\HS(U;E_{\alpha})} &= 0,
    \end{align*}
    and
    \begin{equation*}
      M:= \sup_{n\in\bar \N}\left( \norm{B_n}{\Lip(E_{\alpha}; E)} +
        \norm{C_n}{\Lip(E_{\alpha};\HS(U;E_\alpha))}\right) <\infty,
    \end{equation*}
    it holds that $\lim_{n\to \infty} \sS(X_{0,n},B_n,C_n) = \sS(X_{0,\infty},B_\infty,
    C_\infty)$.
  \end{enumerate}
\end{thm}
\begin{proof}
  Note that the integral equation~\eqref{eq:see_mild} admits a
  solution which admits a continuous modification $X$, see
  \cite[Theorem 3.9]{keller2016stefan} or in a Banach space
  framework~\cite[Theorem 6.2]{van2008stochastic}. Moreover, $X$
  is unique, up to changes on sets of measure zero, among all predictable processes
  $Y\colon \Omega\times [0,T]\to E_{\alpha}$ such that 
  \[ \sup_{0\leq t\leq T} \E{\norm{Y(t)}{E_{\alpha}}^q} <\infty.\]  
  Moreover, \cite[Theorem 3.9]{keller2016stefan} also tells us that 
  \[ \norm{X}{\sL^q_\FF(E_{\alpha})} <\infty\]
  so that $X\in \sL^q_\FF(E_{\alpha})$. 
  
  Let 
  $X_0$, $\widetilde X_0 \in L^q(\Omega,\F_0,\PP; E_\alpha)$,
  $B$, $\widetilde B \in \Lip(E_\alpha;E)$,
  $C$, $\widetilde C \in \Lip(E_\alpha; \HS(U;E_\alpha))$,
  and let $X := \sS(X_0,B, C)$, $\widetilde X:= \sS(\widetilde X_0, \widetilde
  B, \widetilde C)$, then using Lemma~\ref{lem:StHa}, Jensen's inequality and Burkholder-type
  inequality for stochastic convolutions, see~\cite[Theorem
  1.1]{prato1992note}, we get constants $K_q$, $K_{\alpha,S, T}\in(0,\infty)$ such
  that for all $r\in (0,T]$,
  \begin{align*}
    \hspace*{.5em}&\hspace*{-.5em}
                   \E{\sup_{0\leq t\leq r}\norm{X(t)- \widetilde X(t)}{E_\alpha}^q}\\
                 &\leq 3^{q-1}\norm{S}{B([0,T];E_\alpha)}^q \norm{X_0 - \widetilde
                   X_0}{L^q(\Omega,\F_0,\PP; E_\alpha)}^q\\
                 &\quad + 3^{q-1} K_{\alpha, S, T}^q 
                   \E{\sup_{0\leq t\leq r} \left(\int_0^t\norm{B(X(s)) - \widetilde B(\widetilde
                   X(s))}{E} (t-s)^{-\alpha}\d s\right)^q} \\
                 &\quad + 3^{q-1} \int_0^t
                   \E{\sup_{0\leq t\leq r} \norm{\int_0^t S_{t-s}(C(X(s)) - \widetilde C(\widetilde
                   X(s)))\d W_s}{E_\alpha}^q} \\     
                 &\leq 3^{q-1}\norm{S}{B([0,T];E_\alpha)}^q \norm{X_0 - \widetilde
                   X_0}{L^q(\Omega,\F_0,\PP; E_\alpha)}^q\\
                 &\quad + \left(\frac{6}{1-\alpha}\right)^{q-1} \! T^{(1-\alpha) (q-1)}
                   K_{\alpha, S, T}^q 
                   \E{ \sup_{0\leq t\leq r} \int_0^t \norm{B(X(s)) - \widetilde B(X(s))}{E}^q (t-s)^{-\alpha}\d s} \\
                 &\quad + \left(\frac{6}{1-\alpha}\right)^{q-1} \! T^{(1-\alpha) (q-1)}
                   K_{\alpha, S, T}^q 
                   \E{\sup_{0\leq t\leq r} \int_0^t \norm{\widetilde B(X(s)) - \widetilde B(\widetilde X(s))}{E}^q (t-s)^{-\alpha}\d s} \\
                 &\quad + 6^{q-1} K_{q} \norm{S}{B([0,T];E_\alpha)}^q T^{q/2-1}
                   \E{ \int_0^r \norm{(C(X(s)) - \widetilde C(
                   X(s)))}{\HS(U;E_{\alpha})}^q \d s }\\
                 &\quad + 6^{q-1} K_{q}\norm{S}{B([0,T];E_\alpha)}^q T^{q/2-1}
                   \E{ \int_0^r \norm{(\widetilde C(X(s)) - \widetilde C(\widetilde
                   X(s)))}{\HS(U;E_{\alpha})}^q \d s}.
  \end{align*}  
  Now, note that for each $s\in (0,r)$, and $r\in (0,T]$,
  \begin{equation*}
    1 = s^{\alpha} s^{-\alpha}\leq r^\alpha s^{-\alpha}
  \end{equation*}
  and recall Lemma~\ref{lem:sup_conv} and Fubini's theorem to obtain a constant
  $\widetilde K_{\alpha, p, T, S}>0$ such that
  \begin{align*}
    \hspace*{2em}&\hspace*{-2em}
                   \E{\sup_{0\leq t\leq r}\norm{X(t)- \widetilde X(t)}{E_\alpha}^q}\\
                 &\leq 3^{q-1}\norm{S}{B([0,T];E_\alpha)}^q \norm{X_0 - \widetilde
                   X_0}{L^q(\Omega,\F_0,\PP; E_\alpha)}^q\\
                 &\qquad + \widetilde K_{\alpha,p,T, S}
                   \int_0^T \E{  \sup_{0\leq t\leq s}\norm{B(X(t)) -
                   \widetilde B(X(t))}{E}^q} (T-s)^{-\alpha}\d s \\
                 &\qquad + \widetilde K_{\alpha,p,T, S}
                   \int_0^T \E{  \sup_{0\leq t\leq s}\norm{C(X(t)) -
                   \widetilde C(X(t))}{\HS(U;E_\alpha)}^q} (T-s)^{-\alpha}\d s \\
                 &\qquad + \widetilde K_{\alpha,p,T, S}
                   \left(\snorm{\widetilde B}{\Lip(E_\alpha;E)}^q +                   
                   \snorm{\widetilde C}{\Lip(E_{\alpha;E};
                   \HS(U;E_{\alpha}))}^q\right)\\
                 &\qquad\quad \times
                   \int_0^r  \E{\sup_{0\leq t\leq s}\norm{X(t)-
                   \widetilde X(t)}{E_\alpha}^q} (r-s)^{-\alpha}\d s
  \end{align*}      
  Let $M>0$ be such that $\snorm{\widetilde B}{\Lip} + \snorm{\widetilde
    C}{\Lip}\leq M$. 
  Hence, by Gronwall's lemma, see Lemma~\ref{lem:gronwall},
  there exists a constant $K_{\alpha, p, T, S, M}$, depending on $\alpha, p, T, S, M$ such that
  \begin{equation}
    \label{eq:sol_map_continuity}
    \begin{aligned}
      \hspace{1em}&\hspace{-1em}
      \norm{X-\widetilde X}{\sL^q_\FF(E_\alpha)}^q \\
      &\leq K_{\alpha, p, T, S, M}^q  \norm{X_0 - \widetilde X_0}{L^q(\Omega,\F_0,\PP; E_\alpha)}^q\\
      &\quad + K_{\alpha, p, T, S, M}^q 
      \int_0^T \E{  \sup_{0\leq t\leq s}\norm{B(X(t)) -
          \widetilde B(X(t))}{E}^q} (T-s)^{-\alpha}\d s \\
      &\quad + K_{\alpha, p, T, S, M}^q 
      \int_0^T \E{  \sup_{0\leq t\leq s}\norm{C(X(t)) -
          \widetilde C(X(t))}{\HS(U;E_\alpha)}^q} (T-s)^{-\alpha}\d s.
    \end{aligned}
  \end{equation}
  This yields
  \begin{multline*}
    \norm{X-\widetilde X}{\sL^q_\FF}  \leq K_{\alpha, p, T, S, M}
    \left( \norm{X_0 - \widetilde X_0}{L^q}
      \right. \\ \left.+ (T^{1-\alpha}/(1-\alpha))^{1/q}\sup_{x\in E_{\alpha}}\norm{B(x)-\widetilde
          B(x)}{E} 
  \right.\\
   \left.+(T^{1-\alpha}/(1-\alpha))^{1/q}\sup_{x\in E_{\alpha}}
     \norm{C(x)-\widetilde
          C(x)}{\HS(U;E_{\alpha})}\right),
  \end{multline*}
  which by definition of $\norm{.}{\LipBd}$ finishes the proof of~\ref{i:continuity_lip}.

  To prove the strong continuity claim~\ref{i:continuity_strong}, let
  $X_{0,n}$, $B_n$, $C_n$ such that the conditions of
  item~\ref{i:continuity_strong} are fulfilled, and set 
  $X_n = \sS(X_{0,n},B_n,C_n)$, $n\in \bar \N$. 
  For each
  $\omega\in \Omega$ define 
  $\mathcal K(\omega):= X_\infty([0,T])$. By continuity of $X_\infty$,
 $\mathcal K(\omega)\subset E_\alpha$ is compact for
  almost all $\omega\in \Omega$. Hence, by Lemma~\ref{lem:unif_convergence},
  \begin{equation}
    \label{eq:asconv}
    \P{\lim_{n\to\infty} \sup_{x\in \mathcal K} \left(\norm{B_\infty(x) -
          B_n(x)}{E} + \norm{C_\infty(x) - C_n(x)}{\HS(U;E_{\alpha})}\right)=0}= 1.
  \end{equation}
  By linear growth of Lipschitz continuous functions we also get,
  \begin{align*}
    \hspace*{2em}&\hspace*{-2em}
                   \E{\sup_{n\in \N}\sup_{0\leq t\leq T}\norm{B_\infty(X_\infty(t)) -
                   B_n(X_\infty(t))}{E}^q}\\
                 &\leq 2^{q-1}\E{\sup_{n\in\N}\sup_{0\leq t\leq
                   T}\left(\norm{B_\infty(X_\infty(t))}{E}^q +
                   \norm{B_n(X_\infty(t))}{E}^q\right)}\\
                 &\leq 2^{q}\E{\sup_{n\in\bar\N}\sup_{0\leq t\leq T}\norm{B_n(X_\infty(t))}{E}^q}\\
                 &\leq
                   4^{q}\left(\sup_{n\in\bar\N}\norm{B_n}{\Lip(E_\alpha;E)}^q\right)\left(1+\norm{X_\infty}{\sL^q_\FF}^q\right)<\infty. 
  \end{align*}
  and on the same way we get
  \begin{align*}
    \hspace*{2em}&\hspace*{-2em}
    \E{\sup_{n\in \N}\sup_{0\leq t\leq T}\norm{C_\infty(X_\infty(t)) -
        C_n(X_\infty(t))}{\HS(U;E_{\alpha})}^q}\\
                 &\leq
                   4^{q}\left(\sup_{n\in\bar\N}\norm{C_n}{\Lip(E_\alpha;\HS(U;E_{\alpha}))}^q\right)\left(1+\norm{X_\infty}{\sL^q_\FF}^q\right)<\infty. 
  \end{align*}
  Let 
  \[M:=\sup_{n\in \bar \N} \left(\snorm{B_n}{\Lip(E_\alpha;E)} +
      \snorm{C_n}{\Lip(E_\alpha;\HS(U;E_\alpha))}\right) <\infty.\]
  Then,~\eqref{eq:sol_map_continuity} holds true with $X_0:=
  X_{0,\infty}$, $B:= B_\infty$, $C:= C_\infty$, and $\widetilde X_0 :=
  X_{0,n}$, $\widetilde B =
  B_n$, $\widetilde C:= C_n$, for each finite $n\in \N$.
  Hence,~\eqref{eq:asconv} and dominated convergence theorem yield that 
  \begin{align*}
    &
                   \lim_{n\to\infty}\norm{X_\infty - X_n}{\sL^q_\FF}^q\\
                 &\leq K_{\alpha, p, T, S, M}^q  \lim_{n\to\infty} \norm{X_{0,\infty} -  X_{0,n}}{L^q(\Omega,\F_0,\PP; E_\alpha)}^q\\
                 &\quad + K_{\alpha, p, T, S, M}^q \lim_{n\to \infty}
                   \int_0^T \E{\sup_{0\leq t\leq s}\norm{B_\infty(X_\infty(t)) -
                   B_n(X_\infty(t))}{E}^q} (T-s)^{-\alpha}\d s \\
                 &\quad + K_{\alpha, p, T, S, M}^q \lim_{n\to\infty} 
                   \int_0^T \E{  \sup_{0\leq t\leq s}\norm{C_\infty(X_\infty(t)) -
                   C_n(X_\infty(t))}{\HS(U;E_\alpha)}^q} (T-s)^{-\alpha}\d s\\
                 &=0.\qedhere
  \end{align*}
\end{proof}
The following result is now a combination of the previous theorem with
localization of vector-valued stochastic processes. For details on the
localization we refer to~\cite[Section
3.3]{keller2018forward} and \cite[Section 2]{kunze2012continuous}.
\begin{thm}\label{thm:approx_loc}
  Assume the Setting~\ref{sec:sse_setting} and let $q\in
  (2,\infty)$, $\alpha \in [0,1)$,
  for $n\in \bar\N$ let $X_{0,n} \in L^q(\Omega,\F_0,\PP;E_\alpha)$, $B_n\in \LipLoc(E_{\alpha};E)$, $C_n \in
  \LipLoc(E_{\alpha};\HS(U;E_{\alpha}))$, such that for each $x\in E_{\alpha}$,
  \begin{equation}
    \label{eq:loc_conv}
    \begin{aligned}
    \lim_{n\to\infty}
    \norm{X_{0,\infty}-X_{0,n}}{L^q(\Omega;E_{\alpha})} &= 0, \\
    \lim_{n\to\infty} \norm{B_\infty(x) - B_n(x)}{E} &= 0,\\
    \lim_{n\to\infty} \norm{C_\infty(x) - C_n(x)}{\HS(U;E_\alpha)} &= 0,      
    \end{aligned}
  \end{equation}
  and for each $r\in (0,\infty)$, 
  \begin{equation}
    \label{eq:loc_bd}
    M_r :=\sup_{n\in \bar\N} \left(\snorm{B_n}{\Lip(E_{\alpha};E);r} +
      \snorm{C_n}{\Lip(E_{\alpha};\HS(U;E_{\alpha}));r}\right) <\infty.
  \end{equation}
  Then, the following holds true.
  \begin{enumerate}[label=(\roman*)]
  \item \label{i:existence_loc} There exist up to modifications unique
    maximal stopping times $\tau_n$ and unique continuous
    stochastic processes $X_n\colon \llbrak 0,\tau_n\llbrak \to
    E_\alpha$, such that on $\llbrak 0,\tau_n\llbrak$, 
    \begin{equation*}
      X_n(t) = S_t X_{0,n} + \int_0^t
        S_{t-s} B_n(X_n(s)) \d s
      + \int_0^t S_{t-s} C_n(X_n(s)) \d W_s,
    \end{equation*}
    and, in addition, it holds on $\{\tau<\infty\}$ almost surely,
    \begin{equation*}
      \lim_{t\nearrow \tau}\norm{X(t)}{E_{\alpha}} = \infty.
    \end{equation*}
  \item \label{i:conv_exittimes} For $r>0$, the exit times
    \begin{align*}
      \varsigma_n^{(r)} &:= \inf\{t\geq 0\,\colon
                          \norm{X_n(t)}{E_\alpha} \geq r\},\\
      \tau_n^{(r)}  &:= \inf\{t\geq 0\,\colon
                          \norm{X_n(t)}{E_\alpha} > r\},
    \end{align*}
    satisfy for each $r$, $\epsilon \in (0,\infty)$, that almost surely
    \begin{equation*}
      \liminf_{n\to\infty} \tau_n^{(r)} \leq \tau_{\infty}^{(r)} \leq
      \varsigma_\infty^{(r+\epsilon)}\leq \limsup_{n\to\infty} \varsigma_{n}^{(r+\epsilon)},
    \end{equation*}
    and, in particular
    \begin{equation*}
      \lim_{\Q\ni r\to\infty} \liminf_{n\to\infty} \tau_{n}^{(r)}\leq
      \tau_\infty \leq \lim_{\Q\ni r\to\infty} \limsup_{n\to\infty} \varsigma_{n}^{(r)}.
    \end{equation*}
  \item \label{i:continuity_loc} For each $r$, $\epsilon \in (0,\infty)$, it holds that
    \begin{equation}
      \label{eq:continuity_Lp_loc}
      \lim_{n\to\infty}\E{\sup_{0\leq t\leq T}\norm{X_\infty(t\wedge
          \tau_\infty^{(r)}) - X_n(t\wedge \tau_\infty^{(r)} \wedge \varsigma_n^{(r+\epsilon)})}{E_\alpha}^q} = 0,
    \end{equation}
    and for all $t\in [0,T]$, in probability
    \begin{equation}
      \label{eq:continuity_P_loc}
      \lim_{n\to\infty} X_{n}(t)\1_{\llbrak 0,\tau_n \wedge
        \tau_\infty\llbrak}(t;\cdot) = X_\infty(t) \1_{\llbrak 0,\tau_\infty\llbrak}(t;\cdot).
    \end{equation}
  \item \label{i:tau=sigma}
    If for some $r>0$ it holds that
    \begin{equation*}
      \P{\tau_\infty^{(r)} = \varsigma_\infty^{(r)}} = 1,
    \end{equation*}
    then it holds in probability that $\lim_{n\to\infty}\tau_n^{(r)}
    =\tau_\infty^{(r)}$ and, 
    \begin{equation*}
      \lim_{n\to\infty} \sup_{0\leq t\leq T}\norm{X_n(t \wedge \tau_n^{(r)})-  X_\infty(t\wedge \tau^{(r)}_\infty)}{E_{\alpha}} = 0.
    \end{equation*}
  \item \label{i:globalex}
    If, in addition, it holds that for each $n\in \bar \N$,
    \begin{equation*}
      \sup_{x\in E_{\alpha}} \frac{\norm{B_n(x)}{E} +
        \norm{C_n(x)}{\HS(U;E_{\alpha})}}{1+\norm{x}{E_\alpha}} <\infty,
    \end{equation*}
    then, $\tau_n = \infty$ for all $n\in \bar \N$ almost surely and, in probability,
    \begin{equation*}
      \lim_{n\to\infty} \sup_{0\leq t\leq T}\norm{X_n(t) - X_\infty(t)}{E_\alpha} = 0.
    \end{equation*}
  \item \label{i:globalex_unif}
   If, in addition, it holds that
    $\sup_{n\in\N}\E{\norm{X_{n,0}}{E_\alpha}^q}<\infty$ and
    \begin{equation*}
      \sup_{n\in\N}\sup_{x\in E_{\alpha}} \frac{\norm{B_n(x)}{E} +
        \norm{C_n(x)}{\HS(U;E_{\alpha})}}{1+\norm{x}{E_\alpha}} <\infty,
    \end{equation*}
    then, $\tau_n = \infty$ for all $n\in \bar \N$ almost surely, it
    holds that
    \begin{equation*}
      \sup_{n\in \N} \E{\sup_{0\leq t\leq T}
        \norm{X_n(t)}{E_{\alpha}}^q} <\infty,
    \end{equation*}
    and, moreover, for all $p\in [1,q)$,
    \begin{equation*}
      \lim_{n\to\infty} \E{\sup_{0\leq t\leq T} \norm{X_\infty(t) -
          X_n(t)}{E_{\alpha}}^p} = 0.
    \end{equation*}
  \end{enumerate}
\end{thm}
\begin{proof}
  Item~\ref{i:existence_loc} follows for each $n\in\bar\N$ from \cite[Theorem
  3.17]{keller2016stefan} with $E_1$ replaced by $E_\alpha$, cf. Remark~\ref{rmk:interpol:reiteration}. Now, for each $r\in(0,\infty)$ let $h_r\in
  C^\infty([0,\infty))$ be monotonously decreasing functions and such that 
  \begin{equation}
    \label{eq:31}
    c_h := \sup_{r\in [0,\infty)} \norm{h_r'}{\infty} <\infty, 
  \end{equation}
  and $h_r(x) = 1$, for $x \in [0,r^2]$, and $h_r(x) = 0$, for $x\in
  [(r+1)^2,\infty)$. Then define for $r\in(0,\infty)$, $n\in \bar \N$,
  \begin{equation}
    \label{eq:truncation}
    B_{n}^{(r)} := h_r(\norm{\cdot}{E_\alpha}^2) B_n,\qquad C_n^{(r)} :=
    h_r(\norm{\cdot}{E_{\alpha}}^2) C_n.
  \end{equation}
  From \cite[Lemma 3.28 and Lemma 3.29]{keller2018forward} we derive
  that $B_n^{(r)}\in \LipBd(E_{\alpha}; E)$ and $C_{n}^{(r)}\in
  \LipBd(E_{\alpha}; \HS(U;E_{\alpha}))$. We moreover get from these
  lemmas that
  \begin{equation*}
    \begin{aligned}
      \snorm{B^{(r)}_n}{\Lip(E_{\alpha},E)} &\leq
      \snorm{B_n}{\Lip(E_{\alpha};E),r+1} + 2c_h(r+1) \sup_{\substack{x\in {E_{\alpha}}\\ \norm{x}{E_\alpha}\leq r+1}}\norm{B_n(x)}{E} \\
      &\leq 
      \snorm{B_n}{\Lip(E_{\alpha};E),r+1} + 2c_h(r+1)
      (\norm{B_n(0)}{E} + \snorm{B_n}{\Lip(E_\alpha;E);r+1}),
    \end{aligned}
  \end{equation*}
  and the same estimate for $C_n$, which yield with~\eqref{eq:loc_conv} for $x=0$ and \eqref{eq:loc_bd} that
  \begin{equation*}
    \sup_{n\in \bar \N} \snorm{B^{(r)}_n}{\Lip(E_{\alpha};E)} +
    \snorm{C^{(r)}_n}{\Lip(E_{\alpha};\HS(U;E_{\alpha}))} < \infty.
  \end{equation*}
  Let $\sS$ be as defined in
  Theorem~\ref{thm:approx} and set $X^{(r)}_n := \sS(X_0, B_n^{(r)}, C_n^{(r)})$ for $r\in (0,\infty)$
  and $n\in\bar \N$. We have shown that the assumptions of
  Theorem~\ref{thm:approx}.\ref{i:continuity_strong} are satisfied for
  each $r\in(0,\infty)$ and thus 
  \[ \lim_{n\to\infty}\E{\sup_{0\leq t\leq T}\norm{X^{(r)}_\infty(t)-X^{(r)}_n(t)}{E_{\alpha}}^q} = 0.\]
  It remains to relax the truncation. Note that $B_n^{(r)} = B_n$ and $C_n^{(r)} = C_n$ on $\{x\in
  E_{\alpha} \colon \norm{x}{E_{\alpha}} \leq r\}$, and thus, from the
  uniqueness in~\ref{i:existence_loc} we get that 
  \[ X_n^{(r)} = X_n \quad \text{on }\llbrak 0,\tau^{(r)}_n\rrbrak.\] 
  Thus, \cite[Proposition 3.23]{keller2018forward} (with $V:=E_\alpha$, $Y_n:=X_n$), see also
  \cite[Theorem 2.1]{kunze2012continuous}, yields
  \ref{i:conv_exittimes} and, moreover, that for each
  $\epsilon >0$ in probability
  \begin{equation}
    \label{eq:conv_ucp}
    \lim_{n\to\infty} \sup_{0\leq t\leq T}\norm{X_\infty(t\wedge
      \tau_\infty^{(r)}) - X_n(t\wedge \varsigma_n^{(r+\epsilon)}\wedge
      \tau_\infty^{(r)})}{E_\alpha}=0.
  \end{equation}
  In addition,
  \begin{equation*}
    \E{\sup_{n\in \N}\sup_{0\leq t\leq T}\norm{X_\infty(t\wedge
        \tau_\infty^{(r)}) - X_n(t\wedge \varsigma_n^{(r+\epsilon)}\wedge
        \tau_\infty^{(r)})}{E_\alpha}^q} \leq (2r+\epsilon)^q,
  \end{equation*}
  so that by dominated convergence we also
  get~\eqref{eq:continuity_Lp_loc}. Moreover,~\eqref{eq:continuity_P_loc}
  follows now by \cite[Theoerem
  2.1.(3)]{kunze2012continuous}. Item~\ref{i:tau=sigma} follows from
  \cite[Proposition 3.26]{keller2018forward}. 
  Item~\ref{i:globalex} follows from \cite[Corollary
  2.5]{kunze2012continuous}, since, in fact, $\tau_n=\infty$ almost
  surely by linear growth of the coefficients, see \cite[Corollary 3.20]{keller2016stefan}.

  To prove~\ref{i:globalex_unif}, first note that $Y := \sS(0,0,0)$
  satisfies $Y(t)=0$ for all $t\in [0,\infty)$. Write
  \begin{equation}
    \label{eq:38}
    M_{B,C}:=  \sup_{n\in\N}\sup_{x\in E_{\alpha}} \frac{\norm{B_n(x)}{E} +
      \norm{C_n(x)}{\HS(U;E_{\alpha})}}{1+\norm{x}{E_\alpha}}.
  \end{equation}
  For each $r>0$ and $n\in\bar \N$ apply~\eqref{eq:sol_map_continuity}
  to $B:= B_n^{(r)}$, $C:= C_n^{(r)}$ and $\widetilde B:= 0$,
  $\widetilde C:=0$, which yields (with $M=0$), 
  \begin{equation*}
    \begin{aligned}
      \hspace{1em}&\hspace{-1em}
      \norm{X_n^{(r)}}{\sL^q_\FF(E_\alpha)}^q \\
      &\leq K_{\alpha, p, T, S, 0}^q  \norm{X_{n,0}}{L^q(\Omega,\F_0,\PP; E_\alpha)}^q\\
      &\quad + K_{\alpha, p, T, S, 0}^q 
      \int_0^T \E{\sup_{0\leq t\leq s}\norm{B_n^{(r)}(X_n^{(r)}(t))}{E}^q} (T-s)^{-\alpha}\d s \\
      &\quad + K_{\alpha, p, T, S, 0}^q 
      \int_0^T \E{\sup_{0\leq t\leq s}\norm{C_n^{(r)}(X_n^{(r)}(t))}{\HS(U;E_\alpha)}^q} (T-s)^{-\alpha}\d
      s\\
      &\leq K_{\alpha, p, T, S, 0}^q  \norm{X_{n,0}}{L^q(\Omega,\F_0,\PP; E_\alpha)}^q\\
      &\quad + 2^{q} K_{\alpha, p, T, S, 0}^q M_{B,C}^q \left(T^{1-\alpha}/(1-\alpha) + 
        \int_0^T \E{  \sup_{0\leq t\leq s}\norm{X_n^{(r)}(t)}{E_{\alpha}}^q} (T-s)^{-\alpha}\d s\right),         
    \end{aligned}
  \end{equation*}
  so that Grownall's Lemma~\ref{lem:gronwall} yields
  \begin{equation*}
    \sup_{n\in \bar \N} \sup_{r>0}
    \frac{\norm{X_n^{(r)}}{\sL^q_\FF(E_\alpha)}}{1+\norm{X_{0,n}}{L^q(\Omega;E_{\alpha})}}
    <\infty.
  \end{equation*}
  Finally, since $\tau_n = \infty$ almost surely for all $n\in\bar \N$, this yields with Fatou's lemma
  \begin{equation*}
    \sup_{n\in \bar \N} \E{\sup_{0\leq t\leq
        T}\norm{X_n(t)}{E_\alpha}^q}\leq \sup_{n\in \bar \N}
    \liminf_{r\to \infty} \norm{X_n^{(r)}}{\sL^q_\FF(E_\alpha)}^q<\infty.
  \end{equation*}
  Finally, the convergence in probability in~\ref{i:globalex} and uniform boundedness in $\sL^q_\FF$ yield
  convergence in $\sL^p_\FF$, for $p\in [1,q)$. This finishes the
  proof of~\ref{i:globalex_unif}.
\end{proof}


%% file: approx.tex
\section{Approximation of the fixed boundary problems}
\label{sec:proofs}
Throughout this section we use the notation and setup from
Section~\ref{sec:main}. As in \cite{keller2016stefan}, we rewrite the
systems of SPDEs~\eqref{eq:spde} together with the respective interface
dynamics as stochastic evolution equations on the spaces
\[ \L^2:= L^2(\R_+)\oplus L^2(\R_+) \oplus \R,\quad  \H^\alpha :=
  H^{\alpha}(\R_+) \oplus H^{\alpha}(\R_+) \oplus \R,\]
for $\alpha\in [0,\infty)$ and 
\[\H^2_D :=
  H^{2}_D(\R_+) \oplus H_D^{\alpha}(\R_+) \oplus \R,\]
where 
\[H^{2}_D(\R_+) := H^2(\R_+) \cap H^1_0(\R_+).\]
Let $\Delta_D\colon H^2_D(\R_+) \subset L^2(\R_+)\to L^2(\R_+)$ be the
Dirichlet Laplacian and let $A\colon \H^2_D \subset \L^2 \to \L^2$ be the linear operator
defined by
\begin{align*}
  A &:=
      \begin{pmatrix}
        \eta_+ \Delta_D &0&0\\
        0&\eta_-\Delta_D&0\\
        0&0&0
      \end{pmatrix} - \Id_{\L^2},
\end{align*}
and for $u=(u_1,u_2,p)\in \H^{2}$, define
\begin{align*}
  \begin{gathered}[keller]
    N_\mu(u) := 
    \begin{pmatrix}
      \mu_1(\cdot, u_1(\cdot), \ddx u_1(\cdot))\\
      \mu_2(\cdot, u_2(\cdot), \ddx u_2(\cdot))\\
      0
    \end{pmatrix}, \qquad  \bnabla u :=
    \begin{pmatrix}
      \ddx u_1\\ -\ddx u_2 \\ 1
    \end{pmatrix}  
    \\
    \Psi_n(u) := \varrho\left(2n^2\int_0^{1/n} u_1(y) \d y, -2n^2
      \int_0^{1/n} u_2(y) \d y\right),  \quad n\in \N,\\
    \Psi_\infty(u) := \varrho\left(\ddx u_1(0), -\ddx u_2(0)\right),
  \end{gathered}
\end{align*}
for $n\in \bar \N$ let $B_n := N_\mu + \Psi_n(\cdot) \bnabla(\cdot)
+\Id_{\L^2}$ and for $u=(u_1,u_2,p)\in \L^2$ and $w\in U:= L^2(\R)$,
\begin{equation*}
  (C(u)w) =
  \begin{pmatrix}
    \sigma_1(\cdot,u_1(\cdot)) (T_\zeta w)(p+\cdot)\\
    \sigma_2(\cdot,u_2(\cdot)) (T_\zeta w)(p-\cdot)\\
    0
  \end{pmatrix}.
\end{equation*}

We keep the following result on $A$, which is proven in for instance in~\cite[Lemma
4.1 and Lemma 4.2]{keller2016stefan}.
\begin{lem}\label{lem:A}
  The linear operator $A$ on $\L^2$, with $\dom(A):= \H^2_D$ is negative self-adjoint and, in
  particular, generates a strongly continuous analytic semigroup of
  contractions. Moreover, up to equivalence of norms for $\alpha \in
  [0,1/4)$,
  \begin{equation*}
    \dom((-A)^{\alpha}) = \H^{2\alpha}.
  \end{equation*}
\end{lem}

We stress that $\dom(A)$ is a closed subset of $\H^2$ and the norms
$\norm{\cdot}{A}:= \norm{A(\cdot)}{\L^2}$ and $\norm{(\cdot)}{\H^2}$ are
equivalent. In the following We will use the constant
\begin{equation*}
  K_A := \sup_{\substack{u\in \dom(A),\\ u\neq 0}}
  \frac{\norm{u}{\H^2}}{\norm{u}{A}} <\infty.
\end{equation*}

\subsection{Existence and approximation results}

Recall the following from \cite{keller2016stefan}.
\begin{lem}\label{lem:std}
  Assume that the Assumptions~\ref{ass:mu}, \ref{ass:sigma}, 
  and \ref{ass:zeta} hold true and let $\alpha \in (0, 1/4)$. Then,
  \begin{enumerate}[label=(\roman*)]
  \item\label{i:std:Nmu} $N_\mu \colon \dom(A) \to \H^{2\alpha}$ is Lipschitz
    continuous on bounded sets,
  \item\label{i:std:nabla} $\bnabla\colon \dom(A) \to \H^{2\alpha}$ is Lipschitz
    continuous,
  \item\label{i:std:C} $C \colon \dom(A) \to \HS(U;\dom(A))$ is Lipschitz continuous
    on bounded sets.
  \end{enumerate}
\end{lem}
\begin{proof}
  Item~\ref{i:std:Nmu} follows from \cite[Theorem 6.7]{keller2016stefan}, and \ref{i:std:C} from \cite[Theorem 7.6]{keller2016stefan}. Finally, \ref{i:std:nabla} is a direct consequence of the definition of the $\H^1$ and $\H^2$-norms, and continuity of the embedding $\H^1\hookrightarrow \H^{2\alpha}$. 
\end{proof}
We now discuss bounds for the interface coefficients.
\begin{lem}\label{lem:intnorm}
  Let $f\in H^2(\R_+)\cap H^1_0(\R_+)$ and $z\in [0,\infty)$. Then,
  \begin{equation*}
    \abs{\int_0^z f(y) \d y} \leq z^2 \norm{f}{H^2(\R_+)}.
  \end{equation*}
\end{lem}
\begin{proof}
  Since $H^2(\R_+)\hookrightarrow BUC^1([0,\infty))$, we can assume
  without loss of generality that $f\in BUC^1([0,\infty))$. Moreover,
  since $f(0) = 0$ we get from fundamental theorem of calculus
  \begin{equation*}
    \begin{aligned}
    \abs{\int_0^{z} f(y) \d y} &= \abs{\int_0^z \int_0^y \ddx f(x) \d x\d y}\\
      &\leq \int_0^z\int_0^y\abs{\ddx f(x)}\d x \d y \\
      &\leq \frac12 z^2 \sup_{x\in [0,\infty)} \abs{f(x)}
      \leq z^2 \norm{f}{H^2(\R_+)}.      
    \end{aligned}
  \end{equation*}
  Here, we used that
  \begin{equation}
    \label{eq:supdxfH2}
    \sup_{x\in [0,\infty)} \abs{f(x)} \leq  2 \norm{f}{H^2(\R_+)}.
  \end{equation}
  In fact, if we additionally assume that $f\in BUC^2(\R_+)$, then for all $x\in \R_+$,
  \begin{multline*}
    \abs{\ddx f(x)} \leq \sqrt{\int_0^{1} \abs{\ddx f(x+y)}^2 \d y}\\
    +
    \sqrt{\int_{0}^{1} y^2 \int_0^1 \abs{\ddxx f(x+\alpha y)}^2
      \d\alpha \d y }
    \leq 2 \norm{f}{H^2(\R_+)},
  \end{multline*}
  and then~\eqref{eq:supdxfH2} extends to all of $H^2$ since $BUC^2(\R_+)\cap
  H^2(\R_+)\subset H^2(\R_+)$ is dense, see also \cite[Lemma 4.2]{kim2012stochastic}.
\end{proof}

Moreover, we get uniform bounds on the Lipschitz constants.
\begin{lem}\label{lem:equilip}
  Let Assumption~\ref{ass:rho} hold true. Then, for all $n\in \bar \N$,
  $\Psi_n\colon \dom(A) \to \R$ is Lipschitz continuous on bounded
  sets. Moreover, for each $r\in (0,\infty)$,
  \begin{equation*}
    \sup_{n\in \N}\snorm{\Psi_n}{\Lip(\dom(A);\R);r} \leq 2 K_A
    \snorm{\varrho}{\Lip(\R^2;\R);2 K_A r}  <\infty.
  \end{equation*}
\end{lem}
\begin{proof}
  First, let us note that $\Psi_\infty\colon \dom(A) \to \R$ is
  Lipschitz continuous on bounded sets, which follows from continuity
  of the trace operator on $\H^2$ and local Lipschitz continuity of
  $\varrho$, see~\cite[Lemma~4.3]{keller2016stefan}. 
  
  Let $r\in (0,\infty)$ and $u$,
  $\tilde u \in \dom(A)$ with $\norm{u}{A}$, $\norm{\tilde u}{A}\leq
  r$. By Lemma~\ref{lem:intnorm}, for all $n\in\N$
  \begin{equation}
    \label{eq:12}
    n^2\abs{\int_0^{1/n}u_{1\slash 2}(y) \d y}  \leq 
    \norm{u_{1\slash 2}}{H^2(\R_+)},
  \end{equation}
  so that 
  \begin{equation}
    \label{eq:18}
    \norm{\left(2n^2\int_0^{1/n}u_{1}(y) \d y, 2n^2\int_0^{1/n} u_2(y)\d y\right)}{\R^2}
    \leq 2 \norm{u}{\H^2} \leq 2 K_A \norm{u}{A}.
  \end{equation}
  Moreover, we get the same estimate with $u$ replaced by $u-\tilde u$. 
  This yields that for $n\in\N$,
  \begin{align*}
    \hspace{1em}&\hspace{-1em}
                  \abs{\Psi_n(u) - \Psi_n(\tilde u)} \\
                &\leq 2n^2 \snorm{\varrho}{\Lip(\R^2;\R); 2K_A r}
                  \norm{\left(\int_0^{1/n} \left(u_1(y) - \tilde
                  u_1(y)\right)\d y, \int_0^{1/n} \left(u_2(y) -
                  \tilde u_2(y)\right) \d y\right)}{\R^2}\\
                &\leq 2 K_A \snorm{\varrho}{\Lip(\R^2;\R); 2K_A r} \norm{u-\tilde u}{A}. \qedhere
  \end{align*}
\end{proof}

From the previous lemma and \cite[Section 4]{keller2016stefan} we get
the following. 
\begin{lem}\label{lem:lg}
  Assume that the Assumptions~\ref{ass:mu}, \ref{ass:sigma},
  \ref{ass:rho}, \ref{ass:zeta} and~\ref{ass:global} hold true and let $\alpha \in
  (0, 1/4)$. Then, 
  \[ \sup_{n\in \N}\sup_{u\in \dom(A)}
    \frac{\norm{B_n(u)}{\H^{2\alpha}} + \norm{C(u)}{\HS(U;\dom(A))}}{1+
      \norm{u}{\dom(A)}} <\infty.\]
\end{lem}
\begin{proof}
  Linear growth of $N_\mu$ and $C$ have been shown in \cite[Lemma
  4.4]{keller2016stefan}. Since $\varrho$ is bounded so is
  $\Psi_n$. More precisely, it holds that
  \[ \sup_{n\in\N}\sup_{u\in \dom(A)} \abs{\Psi_n(u)} \leq
    \norm{\varrho}{\infty}, \]
  and thus 
  \[\sup_{n\in\N}\sup_{u\in \dom(A)} \norm{\Psi(u)\bnabla u}{\H^1}^2 \leq
    \norm{\varrho}{\infty}\left(1 + \norm{u}{\H^2}^2\right)\leq
    \norm{\varrho}{\infty}(1\wedge K_ A)\left(1 +
      \norm{u}{\dom(A)}^2\right).\]
  Continuity of the embedding $\H^1 \hookrightarrow \H^{2\alpha}$ then finishes the
  proof. 
\end{proof}

Finally, we fix the convergence result on $(\Psi_n)$. 
\begin{lem}\label{lem:Psinconv}
  Let Assumption~\ref{ass:rho} hold true. Then, for all $u\in
  \dom(A)$,
  \[ \sup_{n\in \N} \sqrt{n} \norm{B_n(u) - B_\infty(u)}{\H^{1}} \leq \snorm{\varrho}{\Lip(\R^2;\R);2\norm{u}{\H^2(\R_+)}}
    \norm{u}{\H^2}\left(1+\norm{u}{\H^2}\right). \]
\end{lem}
\begin{proof}
  First note that $B_n(u) - B_\infty(u) = \bnabla u
  \left(\Psi_n(u) - \Psi_\infty(u)\right)$ for each $u\in \H^2_D$. 
  
  Let $f\in H^2(\R_+)\cap H^1_0(\R_+)$ and $z\in \R_+$, then by using
  that $f(0) = 0$, fundamental theorem of calculus and Cauchy-Schwartz inequality,
  \begin{align*}
    \abs{2z^{-2}\int_0^zf(y) \d y -\ddx f(0)} 
    &= \abs{2z^{-2}\int_0^z(f(y)- f(0))\d y - \ddx f(0)}\\
    &= \abs{2z^{-2}\int_0^z\int_0^y (\ddx f(x) - \ddx f(0) )\d x \d y} \\
    &= \abs{2z^{-2}\int_0^z\int_0^y \int_0^x\ddxx f(x')\d x' \d x \d y} \\*
    &\leq \int_0^z\abs{\ddxx f(x)}\d x \leq \sqrt{z} \norm{f}{H^2(\R_+)}.
  \end{align*}
  Recall from Lemma~\ref{lem:intnorm} and~\eqref{eq:supdxfH2} also that
  \begin{equation}
    \max\left\{ \abs{\ddx f(0)}, \abs{2n^2 \int_0^{1/n} f(y) \d y}\right\} \leq 2 \norm{f}{H^2(\R_+)}.
  \end{equation}
  Thus, we get for each $u\in \H^2_D$, with $z:= 1/n$, for $n\in \N$,
  \begin{equation*}
    \abs{\Psi_n(u) - \Psi_\infty(u)}    \leq
    \frac{1}{\sqrt{n}}  \snorm{\varrho}{\Lip(\R^2;\R);2\norm{u}{\H^2(\R_+)}}
    \norm{u}{\H^2(\R_+)} .\qedhere
  \end{equation*}
\end{proof}

From Theorem~\ref{thm:approx_loc} with $E:= \H^{2\alpha}$, for some $\alpha
\in (0,1/4)$, we get the
following; see also \cite[Theorem 3.17 and Corollary 3.21]{keller2016stefan} with $E:= \L^2$.
\begin{prop}\label{prop:existence_loc}
  Let Assumptions~\ref{ass:mu},~\ref{ass:sigma},~\ref{ass:rho}
  and~\ref{ass:zeta} hold true, let $p_0\in \R$, $v_0\colon \R\to
  \R$ such that $v_0(p_0+(\cdot))\vert_{\R_+}$, $v_0(p_0-(\cdot))\vert_{\R_+} \in
  H^2_D(\R_+)$. Then, for each $n\in \bar \N$ there exists a unique maximal
  predictable strictly positive stopping time $\tau_n$ and a unique
  $\FF$-adapted and $\H^2$-continuous stochastic process $X_n$ such that on
  $\llbrak 0,\tau_n\llbrak$ as an $\L^2$-integral equation
  \begin{equation}
    \label{eq:strong_see}
    X_n(t) = X_0  + \int_0^t\left[
      AX_n(s) + B_n(X_n(s))\right] \d s + \int_0^t C(X_n(s))\d W_s,   \quad t\geq 0,
  \end{equation}
  with initial data $X_0:= (v_0(p_0+(\cdot))\vert_{\R_+},
  v_0(p_0-(\cdot))\vert_{\R_+}, p_0)$. Moreover,
  almost surely, 
  \begin{equation*}
    \lim_{t\nearrow \tau_n} \norm{X_n(t)}{\H^2_D} = \infty,\qquad
    \text{on }\{ \tau_n <\infty\}.
  \end{equation*}
\end{prop}

Under additional boundedness assumptions we get global
existence. 
\begin{cor}\label{cor:existence_global}
  Let Assumptions~\ref{ass:mu}, \ref{ass:sigma}, \ref{ass:rho}, \ref{ass:zeta} and additionally the linear growth assumptions~\ref{ass:global} hold true, let $p_0\in \R$, $v_0\colon \R\to \R$ such that $v_0(p_0+(\cdot))\vert_{\R_+}$, $v_0(p_0-(\cdot))\vert_{\R_+} \in
  H^\alpha_D(\R_+)$. Then, for each $n\in \bar \N$ there exists a
  unique $\H^2_D$-continuous and $\FF$-adapted stochastic process such
  that for each $t\in [0,\infty)$, the $\L^2$-integral equation 
  \begin{equation}
    \label{eq:strong_see_global}
    X_n(t) = X_0  + \int_0^t\left[
      AX_n(s) + B_n(X_n(s))\right] \d s + \int_0^t C(X_n(s))\d W_s,   
  \end{equation}
  holds true almost surely, with initial data $X_0:= (v_0(p_0+(\cdot))\vert_{\R_+},
  v_0(p_0-(\cdot))\vert_{\R_+}, p_0)$. Moreover, for each $q\in
  [1,\infty)$ and $T\in (0,\infty)$,
  \begin{equation}
    \label{eq:unif_bd}
    \sup_{n\in \N} \E{\sup_{0\leq t\leq T}\norm{X_n(t)}{\H^2_D}^q} < \infty.
  \end{equation}
\end{cor}
\begin{proof}
  It follows from Lemma~\ref{lem:lg} that $C$ and
  $B_n$ have linear growth, uniformly in $n\in \bar \N$. Thus,
  \cite[Corollary 3.20]{keller2016stefan} yields global mild solutions
  on $\H^2_D$, which is by \cite[Corollary 3.21]{keller2016stefan}
  also the unique strong solution and, in particular, satisfies
  \eqref{eq:strong_see_global}. Now, the
  $L^q$-boundedness~\eqref{eq:unif_bd} follows from
  Theorem~\ref{thm:approx_loc}.\ref{i:globalex_unif} .
\end{proof}

\begin{thm}[Approximation Theorem]\label{thm:approx_Stefan_see}
  Let Assumptions~\ref{ass:mu},~\ref{ass:sigma},~\ref{ass:rho}
  and~\ref{ass:zeta} hold true, and denote by $(\tau_n, X_n)$ the
  unique continuous maximal solutions of \eqref{eq:strong_see},
  respectively for each $n\in\bar\N$. Then, the following holds true.
  \begin{enumerate}[label=(\alph*)]
  \item\label{i:conv_see_loc} For each $\alpha\in (0,1/4)$ and $q\in (2,\infty)$, the assumptions of
    Theorem~\ref{thm:approx_loc}.\ref{i:existence_loc} --
    \ref{i:continuity_loc} are satisfied choosing $E:= \H^{2\alpha}$. In particular, for each $t\geq 0$, in probability,
    \[ \H^2-\lim_{n\to \infty} X_n(t) \1_{\llbrak 0,\tau_n\wedge \tau_\infty\llbrak}(t) = X_\infty(t)
      \1_{\llbrak 0,\tau_\infty\llbrak}(t).\]
  \item\label{i:conv_see_global} If, in addition, Assumption~\ref{ass:global} holds true, then
    for each $\alpha \in (0,1/4)$ and $q\in (2,\infty)$ the
    assumptions of Theorem~\ref{thm:approx_loc}.\ref{i:globalex} and
    \ref{i:globalex_unif} are satisfied. In particular, for each $q'\in [1,\infty)$,
    \[ \lim_{n\to\infty} \E{\sup_{0\leq t\leq T}\norm{X_\infty(t) -
          X_n(t)}{\H^2_D}^{q'}} = 0.\]
  \end{enumerate}  
\end{thm}
\begin{proof}
  Part~\ref{i:conv_see_loc} is a consequence of Lemmas~\ref{lem:A},
  \ref{lem:std}, \ref{lem:equilip} and
  \ref{lem:Psinconv}. Item~\ref{i:conv_see_global} then follows by
  additional application of Lemma~\ref{lem:lg}.
\end{proof}

\begin{rmk}\label{rmk:convrate}
  After truncation of the coefficients as in~\eqref{eq:truncation},
  Lemma~\ref{lem:Psinconv} together with Theorem~\ref{thm:approx}.\ref{i:continuity_lip} even yields the convergence rate $1/2$ for the solutions of the corresponding truncated solutions. Then for each stopping time $\tau$ such that
  for some $r\in (0,\infty)$
  \[\tau \leq \inf_{n\in\bar\N}\inf\{ t\geq 0\,\vert\,
    \norm{X_n(t)}{\H^2_D} > r\}, \]
  we get in Theorem~\ref{thm:approx_Stefan_see}.\ref{i:conv_see_global} that
  for all $q'\in (1,\infty)$
  \begin{equation*}
    \sup_{n\in\N} n^{q'/2} \E{\sup_{0\leq t\leq T}\norm{X_\infty(t\wedge \tau) -
        X_n(t\wedge \tau)}{\H^2_D}^{q'}}<\infty. 
  \end{equation*}
  However, since $B_n$, $n\in \bar\N$, are not globally bounded it is not
  immediate to see if this translates to the solutions of the original
  equations. 
\end{rmk}

\subsection{From fixed to moving boundary problem}
To translate the results on the stochastic evolution equations to the
moving boundary problems, we define in a first step the isometric isomorphism
\begin{equation*}
    \iota\colon L^2(\R_+) \oplus L^2(\R_+)\to L^2(\R),\qquad
    (u_1, u_2) \mapsto
    \begin{cases}
      u_1,& \text{on } (0,\infty),\\
      u_2,& \text{on } (-\infty,0).
    \end{cases}    
\end{equation*}
Then, the transformation will be performed by the mapping
\begin{equation}
  \label{eq:trafofct}
  F \colon L^2(\R_+)\oplus L^2(\R_+)\oplus \R \to L^2(\R),\quad
  (u_1,u_2, x)
  \mapsto (\iota(u_1,u_2))(\cdot - x).
\end{equation}
The mappings have the following properties:
\begin{lem}\label{lem:trafo}
  \begin{enumerate}[label=(\alph*)] 
  \item\label{i:trafo} The restriction of $\iota$ to $H^1_0(\R_+)\oplus H^1_0(\R_+)$
    defines an isometric isomorphism into
    \[ \{u\in L^2(\R) \,\vert\,
      u\vert_{\dot \R} \in H^1(\dot \R),\,
      u(0)=0\}\subset H^1(\R)\]
  \item\label{ii:trafo}
    \(\iota((H_D^2(\R_+))^{\times 2}) =  \Gamma(0) =
    H^2_D(\dot\R)\)
  \item\label{iv:trafo} $F\in C(\L^2; L^2(\R))$     
  \item\label{iii:trafo} $F\vert_{\H^1_D}$ defines an element of $C^1(\H^1_D; L^2(\R))$ and of $C(\H^1_D;H^1(\R))$.
  \end{enumerate}
\end{lem}
\begin{proof}
  The first two results are immediate from the definition of
  direct sums of Hilbert spaces and of the Sobolev
  spaces. Moreover, item~\ref{iii:trafo} follows from \cite[Lemma
  5.3.(2)]{keller2016stefan}. To see the last part let for $n\in\bar \N$ $u_n= (u_{n,1}, u_{n,2}, u_{n,3})\in \L^2$ be such that $\lim_{n\to\infty} u_n = u_\infty$. Then,
  \begin{multline*}
    \lim_{n\to\infty} \int_\R \abs{\iota(u_{n,1}, u_{n,2})(x-u_{n,3})- \iota(u_{1,\infty}, u_{2,\infty})(x-u_{\infty,3})}^2 \d x \\
    \leq  2\lim_{n\to\infty} \int_\R \abs{\iota(u_{\infty ,1}, u_{\infty ,2})(x-u_{n,3}) - \iota(u_{\infty,1}, u_{\infty,2})(x-u_{\infty,3})}^2 \d x \\   
   +2 \lim_{n\to\infty} \int_\R \abs{\iota(u_{n,1}, u_{n,2})(x-u_{n,3}) - \iota(u_{\infty,1}, u_{\infty,2})(x-u_{n,3})}^2 \d x.
  \end{multline*}
  Here, on the right hand side, the first limit equals $0$ due to strong continuity of the translation group on $L^2(\R)$, and the second limit vanishes by translation invariance of the Lebesgue measure and continuity of $\iota$. Hence, $F$ is continuous.
\end{proof}

In general, $F$ is not $C^2$. However, the transformation from
Section~\ref{ssec:proofmainres} can be made rigorous by using the
stochastic chain rule from~\cite[Section
5]{keller2016stefan}. The following result then eventually yields Theorem~\ref{thm:existence} and allows to derive Theorem~\ref{thm:conv_smbp} from Theorem~\ref{thm:approx_Stefan_see}. 
\begin{prop}\label{prop:trafo}
  Let Assumptions~\ref{ass:mu},~\ref{ass:sigma},~\ref{ass:rho}
  and~\ref{ass:zeta} hold true and for $n\in \bar \N$ let $\tau_n$,
  and $X_n = (X_{n;1}, X_{n;2}, X_{n;3})$ be given as in
  Proposition~\ref{prop:existence_loc}. On $\llbrak 0,\tau_n\llbrak$,
  define $v_n(t,\cdot):= F(X_n(t))$ and $p_n^*(t) :=
  (X_{n,3}(t))$. Then, in the sense of Definition~\ref{df:smbp}, for
  each $n\in\bar\N$, $(\tau_n, v_n, p_n^*)$ is the unique solution
  of~\eqref{eq:smbp} with interface
  condition~\eqref{eq:Stefan_cond_approx} for
  $n<\infty$, and with interface condition~\eqref{eq:Stefan_cond} for
  $n=\infty$, respectively.  
\end{prop}
\begin{proof}
  By \ref{i:trafo} and \ref{ii:trafo} of Lemma~\ref{lem:trafo},
  $t\mapsto \iota((X_n(t))_1, (X_n(t))_2)$ is $H^1(\R)$-continuous and
  we can apply the stochastic chain rule~\cite[Theorem
  5.4]{keller2016stefan} to 
  \begin{equation}
    (v_t:= \iota(X_{n;1}(t), X_{n;2}(t)), x_t:= X_{n;3}(t), \tau := \tau_n)    
  \end{equation}
  respectively for each $n\in \bar \N$. In particular, since $X_n$
  takes values in $\dom(A)= \H^2_D$, Lemma~\ref{lem:trafo}.\ref{ii:trafo} and
  the definition of $F$ yield that $F(X_n(t))\in \Gamma(X_{n;3}(t))$ on $\llbrak 0,\tau_n\llbrak$. $(\tau_n,v_n,p_n^*)$ is indeed
  a solution in the sense of Definition~\ref{df:smbp}.
  
  The uniqueness claim follows by application of the stochastic chain rule~\cite[Theorem
  5.4]{keller2016stefan} with $(v_t:= \iota^{-1}(v_n(t,\cdot)), x_t :=
  - p_n^*(t), \tau := \tau_n)$, respectively for each $n\in \bar \N$,
  and the uniqueness for the strong integral equation in
  Proposition~\ref{prop:existence_loc}. 
\end{proof}


%% file: notation.tex
\section{Notation}
\label{sec:notation}

For any finite union of disjoint intervals $I\subset \R$ denote by
$H^k(I)$, $k\in \N_0$, the $k$-th order Sobolev space and let
$C_0^\infty(I)$ be the space of smooth functions with compact support in
$I$, $BUC^k$ is the space of bounded uniformly continuous functions
with $k$- bounded uniformly continuous derivatives and let $L^2(I) =
H^0(I)$ and $L^p(I)$, $p\in [1,\infty]$ be the Lebesgue space, and
$L^p_{loc}$ be the spaces of locally $p$-integrable functions. For all
such sets $I\subset \R$ and all $f\in H^2(I)$, we denote by $\nabla f$
and $\Delta f$ respectively the first and second (piecewise) weak derivative on $I$, that is if $I = \bigcup_{k=1}^n I_k$ for disjoint intervals $I_k$, $k=1,...,n$, then for all $\phi \in C^{\infty}_0(I)$,
\begin{align*}
  \int_I \nabla f(x)\phi(x) \d x &= \sum_{k=1}^n \int_{I_k} \nabla f(x) \phi(x) \d x = - \int_{I} \phi'(x) f(x) \d x\\
  \int_I \Delta f(x)\phi(x) \d x &= \sum_{k=1}^n \int_{I_k} \Delta f(x) \phi(x) \d x =  \int_{I} \phi''(x) f(x) \d x.
\end{align*}

Let $\R_+:= (0,\infty)$, $\dot \R := \R \setminus\{0\}$ and $\dot \R_x
:= \R\setminus\{x\}$, for $x\in \R$. 

For Banach spaces $E$, $F$ we say $E\hookrightarrow F$ if $E$ is
continuously embedded into $F$, and we denote by $E\oplus F$ the direct sum,
i.\,e. the Banach space $E\times F$ equipped with the norm 
\[ \norm{(e,f)}{E\oplus F} := \sqrt{\norm{e}{E}^2 +
    \norm{f}{F}^2},\qquad  e\in E, f\in F.\]
We denote by $B(E;F)$ the space of bounded functions mapping $E$ into $F$, equipped with norm
\begin{equation*}
  \norm{\Phi}{B(E;F)} := \sup_{x\in E} \norm{B(x)}{F},\qquad \Phi \in B(E;F),
\end{equation*}
and by $\Lip(E;F)$ the space of
Lipschitz continuous functions mapping $E$ into $F$, equipped with the norm
\begin{equation*}
  \norm{\Phi}{\Lip(E;F)} := \norm{\Phi(0)}{F} +
  \snorm{\Phi}{\Lip(E;F)},\qquad \Phi \in \Lip(E;F),
\end{equation*}
where the seminorm $\snorm{\cdot}{\Lip(E;F)}$ is defined as
\begin{equation*}
  \snorm{\Phi}{\Lip(E;F)} := \sup_{\substack{x,y\in E \\ x\neq y}}
  \frac{\norm{\Phi(x) - \Phi(y)}{F}}{\norm{x-y}{E}},\qquad \Phi \in\Lip(E;F),
\end{equation*}
and denote by $\LipBd(E;F)$ the space of bounded functions in
$\Lip(E;F)$, equipped with norm
\begin{equation*}
  \norm{\Phi}{\LipBd(E;F)}:= \sup_{x\in E} \norm{\Phi(x)}{E} +
  \snorm{\Phi}{\Lip(E;F)},\qquad \Phi \in \LipBd(E;F),
\end{equation*}
Moreover, let $\LipLoc(E;F)$ be the space of functions $\Phi \colon
E\to F$
such that for all $r\in \R_+$, $\Phi \vert_{B_E(r)}$ is Lipschitz
continuous, where $B_E(r):= \{x\in E\colon \norm{x}{E}\leq r\}$, and
we define
\[ \snorm{\Phi}{\Lip(E;F);r} :=\sup_{\substack{x,y\in B_{E}(r) \\ x\neq y}}
  \frac{\norm{\Phi(x) - \Phi(y)}{F}}{\norm{x-y}{E}},\qquad \Phi
  \in\Lip(E;F),\,r>0.\]
Given additionally a filtered probability space $(\Omega,\F, \FF,\PP)$,
let $\sL^0_{\FF}(E)$ be the set of equivalence classes of
$\FF$-adapted and $E$-continuous stochastic processes,
for $q\in [2,\infty)$ and $X= (X_t)_{t\in [0,T]}\in \sL^0_\FF(E)$ define
\begin{equation*}
  \norm{X}{\sL^q_{\FF}(E)} := \left(\E{ \sup_{0\leq t\leq T}\norm{X_t}{E}^q}\right)^{1/q} \in [0,\infty],
\end{equation*}
and set $\sL^q_{\FF}(E):= \{X\in \sL^0_{\FF}(E) \colon 
\norm{X}{\sL^q_{\FF}(E)} <\infty\}$.

For $\R$-valued random variables $\varsigma$ and $\tau$ define the closed stochastic interval
\[ \llbrak \varsigma, \tau \rrbrak := \{(t,\omega) \colon
  \varsigma(\omega)\leq t \leq \tau(\omega)\}.\]
We say that for stochastic processes $X$ and
$Y$ that $X=Y$ on $\llbrak \varsigma, \tau\rrbrak$, if
\[X(t;\omega) = Y(t;\omega)\] for all $t$ and almost all $\omega$ such
that $(t,\omega)\in \llbrak \varsigma,\tau \rrbrak$. In the same way the half opened and open intervals $\llbrak\varsigma,\tau\llbrak$, $\rrbrak \varsigma, \tau \rrbrak$ and $\rrbrak
\varsigma, \tau \llbrak$ are defined.


%% file: 0paper.bbl
\begin{thebibliography}{10}

\bibitem{cont2014price}
R.~Cont, A.~Kukanov, and S.~Stoikov.
\newblock The price impact of order book events.
\newblock {\em Journal of financial econometrics}, 12(1):47--88, 2014.

\bibitem{prato1992note}
G.~da~Prato and J.~Zabczyk.
\newblock A note on stochastic convolution.
\newblock {\em Stochastic Analysis and Applications}, 10(2):143--153, 1992.

\bibitem{engel1999one}
KJ~Engel and R.~Nagel.
\newblock {\em One-parameter semigroups for linear evolution equations}, volume
  194.
\newblock Springer Science \& Business Media, 1999.

\bibitem{hambly2018reflected}
B.~Hambly and J.~Kalsi.
\newblock {A Reflected Moving Boundary Problem Driven by Space-Time White
  Noise}.
\newblock {\em arXiv preprint arXiv:1805.10166}, 2018.

\bibitem{hambly2018stefan}
B.~Hambly and J.~Kalsi.
\newblock {Stefan Problems for Reflected SPDEs Driven by Space-Time White
  Noise}.
\newblock {\em arXiv preprint arXiv:1806.04739}, 2018.

\bibitem{henry2006geometric}
D.~Henry.
\newblock {\em Geometric theory of semilinear parabolic equations}, volume 840.
\newblock Springer, 2006.

\bibitem{keller2016stefan}
M.~Keller-Ressel and MS~M{\"u}ller.
\newblock {A Stefan-type stochastic moving boundary problem}.
\newblock {\em Stochastics and Partial Differential Equations: Analysis and
  Computations}, 4(4):746--790, 2016.

\bibitem{keller2018forward}
M.~Keller-Ressel and MS~M{\"u}ller.
\newblock {Forward-Invariance and Wong-Zakai Approximation for Stochastic
  Moving Boundary Problems}.
\newblock {\em arXiv preprint arXiv:1801.05203}, 2018.

\bibitem{kim2012stochastic}
K~Kim, Z~Zheng, and RB~Sowers.
\newblock {A stochastic Stefan problem}.
\newblock {\em Journal of Theoretical Probability}, 25(4):1040--1080, 2012.

\bibitem{kunze2012continuous}
M.~Kunze and JMAM van Neerven.
\newblock Continuous dependence on the coefficients and global existence for
  stochastic reaction diffusion equations.
\newblock {\em Journal of Differential Equations}, 253(3):1036--1068, 2012.

\bibitem{lipton2014trading}
A.~Lipton, U.~Pesavento, and M.~Sotiropoulos.
\newblock Trading strategies via book imbalance.
\newblock {\em Risk}, page~70, 2014.

\bibitem{lunardi1995analytic}
A.~Lunardi.
\newblock {\em Analytic Semigroups and Optimal Regularity in Parabolic
  Problems}, volume~16.
\newblock Springer Science \& Business Media, 1995.

\bibitem{lunardi2009interpolation}
A.~Lunardi.
\newblock Interpolation theory.
\newblock {\em Lecture Notes. Scuola Normale Superiore di Pisa}, 2009.

\bibitem{mueller2016stochastic}
MS~M{\"u}ller.
\newblock {A stochastic Stefan-type problem under first-order boundary
  conditions}.
\newblock {\em The Annals of Applied Probability}, 28(4):2335--2369, 2018.

\bibitem{stefan1888theorie}
J.~{Stefan}.
\newblock {\"Uber die Theorie der Eisbildung, insbesondere \"uber die
  Eisbildung im Polarmeere.}
\newblock {\em {Wien. Ber.}}, 98:965--983, 1888.

\bibitem{van2008stochastic}
JMAM Van~Neerven, MC~Veraar, and L.~Weis.
\newblock {Stochastic evolution equations in UMD Banach spaces}.
\newblock {\em Journal of Functional Analysis}, 255(4):940--993, 2008.

\bibitem{zheng2012stochastic}
Z.~Zheng.
\newblock {\em {Stochastic stefan problems: Existence, uniqueness, and modeling
  of market limit orders}}.
\newblock University of Illinois at Urbana-Champaign, 2012.

\end{thebibliography}
